\documentclass{mcom-l}

\usepackage{import}
\usepackage{amsfonts}
\usepackage{placeins}
\usepackage{amsmath}
\usepackage{amsthm,amssymb}
\usepackage{fullpage,relsize,cite,graphicx,mathrsfs,comment,caption,subcaption,mathtools}
\usepackage{color}
\usepackage{float}
\usepackage{todonotes}
\usepackage{xifthen}
\usepackage{marginnote}
\newcounter{todocounter}
\DeclareRobustCommand{\MyChange}[3][\empty]{%
  {\color{#2}#3}
  \ifthenelse{\isempty{#1}}{}
  {%
    \addtocounter{todocounter}{1}%
    \ifmmode%
        {\color{#2}\text{$^{\framebox{\arabic{todocounter}}}$}}%
    \else%
        {\color{#2}\text{$^{\arabic{todocounter}}$}}%
    \fi%
    \marginpar{\textcolor{#2}{$^{\arabic{todocounter}}$\textnormal{#1}}}%
  }%
}

\makeatletter
\newcommand{\LeftEqNo}{\let\veqno\@@leqno}
\makeatother

\usepackage{environ}
\makeatletter
\NewEnviron{Lalign}{\tagsleft@true\begin{align}\BODY\end{align}}
\makeatother

\usepackage{tikz}

\definecolor{lightgray}{rgb}{0.7,0.7,0.7}
\usepackage[normalem]{ulem}

\definecolor{ghcol}{rgb}{0.7,0.7,0}

\DeclareMathOperator*{\esssup}{ess\,sup}

\newtheorem{lemma}{Lemma}
\newtheorem{corollary}{Corollary}
\newtheorem{definition}{Definition}
\newtheorem{theorem}{Theorem}
\newtheorem{problem}{Problem}

\newtheorem{proposition}{Proposition}
\newtheorem{remark}{Remark}
\newtheorem{assumption}{Assumption}
\newtheorem{GeoSet}{Geometric setting}

\newcommand{\re}{\mathbb{R}}
\newcommand{\ud}{\underline{D}}

\newcommand{\RN}[1]{%
  \textup{\uppercase\expandafter{\romannumeral#1}}%
} 

\newcommand\irregularcircle[2]{
  \pgfextra {\pgfmathsetmacro\len{(#1)+rand*(#2)}}
  +(0:\len pt)
  \foreach \a in {10,20,...,350}{
    \pgfextra {\pgfmathsetmacro\len{(#1)+rand*(#2)}}
    -- +(\a:\len pt)
  } -- cycle
}

\numberwithin{equation}{section}
\numberwithin{lemma}{section}
\numberwithin{corollary}{section}
\numberwithin{theorem}{section}
\numberwithin{proposition}{section}
\numberwithin{problem}{section}
\numberwithin{remark}{section}
\numberwithin{definition}{section}
\numberwithin{assumption}{section}

\title{A domain mapping approach for elliptic equations posed on random bulk and surface domains
\thanks{ TODO 
}}

\author{Lewis Church}
\address{Mathematics Institute, Zeeman Building, University of Warwick, Coventry. CV4 7AL. UK}
\email{Lewis.Church@warwick.ac.uk}

\author{Ana Djurdjevac}
\address{Institute of Mathematics, Straße des 17. Juni 136, Technical University of Berlin, 10623 Berlin,  Germany}
\email{a.djurdjevac@tu-berlin.de}

\author{Charles M. Elliott}
\address{Mathematics Institute, Zeeman Building, University of Warwick, Coventry. CV4 7AL. UK}
\email{C.M.Elliott@warwick.ac.uk}

\date{}

\begin{document}
\setlength\parindent{0pt}
\begin{abstract}
{
In this article, we analyse the domain mapping method approach to approximate statistical moments of solutions 
to linear elliptic partial differential equations posed over random geometries including smooth surfaces and bulk-surface systems. 
In particular, we present the necessary geometric analysis required by the domain mapping method to reformulate elliptic equations on random surfaces
onto a fix deterministic surface using a prescribed stochastic parametrisation of the random domain.
An abstract analysis of a finite element discretisation coupled with a Monte-Carlo sampling is presented for the resulting 
elliptic equations with random coefficients posed over the fixed curved reference domain and 
optimal error estimates are derived. The results from the abstract framework are applied to a model elliptic problem on a random surface and a coupled 
elliptic bulk-surface system and the theoretical convergence rates are confirmed by numerical experiments.
}
\end{abstract}
\maketitle


\section{Introduction}

In the mathematical characterization of numerous scientific and engineering systems, the topology of the domain may not be precisely described. 
The main sources of uncertainty are usually insufficient data, measurement errors or manufacturing variability.  
This uncertainty in the geometry often naturally appears in many applications including surface imaging, manufacturing of nano-devices, material science and biological systems. 
As a result, the analysis of uncertainty in the computational domain has become an interesting and rich mathematical field. 

A comprehensive summary concerning the first directions in the treatment of elliptic partial differential equations (PDEs) in random domains
can be found in \cite{xiu2006numerical, harbrecht2016analysis, dambrine2016numerical, nouy2008extended, canuto2007fictitious} and recently 
\cite{djurdjevac2018RD} for a parabolic equation on a randomly evolving domain. 
Aside from the fictitious domain method \cite{canuto2007fictitious, nouy2011fictitious, nouy2008extended}, 
the main approaches utilize a probabilistic framework by describing the random boundary of the domain with a random field.
This probabiltistic approach is usually proceeded with one of two main techniques: the perturbation approach and the domain mapping method. 
The perturbation approach \cite{harbrecht2013first, harbrecht2008sparse} 
exploits a shape Taylor expansion with respect to the boundary random field to represent the solution, however as a result 
is limited to consideration of only small random deformations. 
The domain mapping approach \cite{xiu2006numerical, harbrecht2014numerical, castrillon2016analytic} on the other hand allows does not suffer the same limitations.
The key idea behind this  method is to define an extension of the random 
boundary process into the interior domain to form a complete random mapping for the whole domain and then to use this domain mapping 
to transform the original partial differential equation on the random domain onto  
the fixed deterministic reference domain resulting in an partial differential equations with random coefficients. 
For the latter formulation, 
there is a wealth of literature available on numerical techniques to compute any quantities of interest, 
see for example \cite{lord2014introduction, le2010spectral, gunzburger2014stochastic}.
The aim of this paper, is to incorporate the domain mapping method with the well-developed field of surface 
PDEs \cite{dziuk2013finite, elliott2017unified,deckelnick2005computation} which has so far only considered uncertainty in 
the coefficients of the considered PDEs, see \cite{djurdjevac2017evolving}. 
This will lead to more realistic geometric description of many of the situations previously dicussed. 
Note that while the domain mapping method will be applicable to domains with random rough surfaces, we will only choose to focus on sufficiently smooth random surfaces and 
leave the rough case for future considerations.


The layout of the article is as follows. In Section 2, we provide an overview of the domain mapping method for partial differential equations in flat random domains
and furthermore discuss suitable notions for the expectation of a family of random domains.
In Section 3, we introduce the necessary geometric analysis and computations required to apply the domain mapping method to elliptic partial differential equations 
posed on random surfaces. 
In Section 4, we present a model elliptic problem on a random surface and a coupled elliptic system on a random bulk-surface, and analyse weak formulations in 
both the stochastic and spatial variables for the 
reformulated equations with stochastic coefficients on the fixed reference surface and bulk-surface respectively.
Section 5 provides an abstract analysis of a finite element discretisation incorporating a pertubation to the variational set-up due 
to a first order approximation of the curved reference domain, and couples with a Monte-Carlo sampling to approximate the first moment of the solution.
An optimal error estimate is derived and subsequently applied in Section 6, to two discretisations of the proposed reformulated problems. 
%
We conclude in Section 7, by presenting numerical results confirming the 
theoretical rate of  convergence.



\section{The domain mapping method}
We begin with a brief introduction on spaces of random fields.
For further details on these spaces, we refer the reader to \cite{lord2014introduction}.
Note throughout this paper, we will let $(\Omega, \mathcal{F}, \mathbb{P})$ denote a complete, separable probability space consisting of a sample space
$\Omega$, a $\sigma-$algebra of events $\mathcal{F}$ and a probability measure $\mathbb{P}$.

\subsection{Random field notation}
For a given Banach space $V$ and $p\in [1, \infty]$, the Lebesgue-Bochner space $L^p(\Omega; V)$ consists of all strongly $\mathcal{F}-$measurable functions 
$
f: \Omega \rightarrow V
$
for which the norm 
\begin{equation*}
 \|f\|_{L^p(\Omega;V)} = 
 \begin{cases}
  \left( \int_{\Omega} \|f(\omega)\|_{V}^p \, d\mathbb{P}(\omega) \right)^{\frac{1}{p}} \quad &p \in [1, \infty)\\
  \esssup\limits_{\omega} \|f(\omega)\|_{V} \quad &p = \infty,
 \end{cases}
\end{equation*}
is finite. 
For convenience, we will express the parameters of a given random field $(f(\omega))(x)$ by $f(\omega,x)$.
In the case that $V$ is a separable Hilbert space, it follows that $L^2(\Omega; V)$ is also a separable Hilbert space and furthermore is isomorphic 
to the tensor product 
\begin{equation}
 L^2(\Omega;V)  \cong L^2(\Omega)\otimes V.
\end{equation}
For details, see \cite{reedmethods}.

\subsection{The domain mapping method} 
To illustrate the key concepts of the domain mapping method, consider the following boundary value
problem 
\begin{align}\label{eq:ExampleMot}
 - \Delta u(\omega) &= f(\omega) \quad \text{in } D(\omega)\\
 u(\omega) &= 0 \quad \text{on } \Gamma(\omega),\nonumber 
\end{align}
posed on an open, connected, bounded domain $D(\omega)\subset \re^2$ with a random boundary $\Gamma(\omega) = \partial D(\omega)$. Here 
the prescribed random field $f(\omega):D(\omega)\rightarrow \re$ and additionally the boundary,
will be assumed to be sufficiently regular to ensure well-posedness for $a.e.$ $\omega$. 
The first essential feature of the domain mapping method is the representation of the stochastic boundary via a random field.
More precisely, in the above context we will assume that there exists a random field $\phi \in L^{\infty}(\Omega; C^{0}(\Gamma_0; \re^2)),$
that maps a fixed closed curve $\Gamma_0\subset \re^2$ onto realisations of the random boundary 
$ \phi(\omega,\cdot): \Gamma_0 \rightarrow \Gamma(\omega),$
see figure \ref{fig:param}.
The next step in the method is to define an extension of the boundary process into the interior to form a
stochastic mapping
$
 \phi(\omega, \cdot): \overline{D_0} \rightarrow \overline{D(\omega)}
$
for the whole domain. For instance, \cite{xiu2006numerical} proposed an extension based on the solution of the Laplace equation over the unit square with boundary 
conditions prescribed by segments of the random boundary. However, alternative 
approaches may wished to be considered depending on the application in question and the geometry of the computational reference domain.
 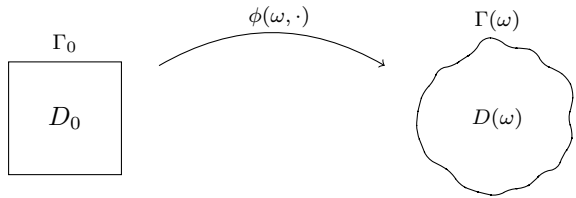
\begin{figure}[H]
\begin{tikzpicture}
  \coordinate (d) at (1.5,0);
  \draw [draw=black] (-5,-0.75) rectangle (-3.5,0.75) node[pos=.5] {$D_0$};
  \draw node[] at (-4.25,1) {\footnotesize $\Gamma_0$};
  \draw node[] at (1.5,1.3) {\footnotesize $\Gamma(\omega)$};
  \draw node[] at (1.5,0) {\footnotesize $D(\omega)$};

  \draw[->] (-3, 0.7)  to [out=30,in=150] (0,0.7);
    \draw node[] at (-1.4,1.35) {\footnotesize $\phi(\omega, \cdot)$};
 
   \draw[black,rounded corners=1mm] (d) \irregularcircle{1cm}{1mm};
\end{tikzpicture} 
  \caption{A realisation of the random domain mapping.}
  \label{fig:param}
\end{figure}
With a complete domain mapping at hand, the random domain problem (\ref{eq:ExampleMot}) can now be reformulated 
as a partial differential equation with random coefficients over the fixed deterministic domain $D_0$,
\begin{align*}
 -\frac{1}{\sqrt{g(\omega)}} \nabla \cdot \left( \sqrt{g(\omega)} G^{-1}(\omega) \nabla ( u\circ \phi ) (\omega) \right) &= (f\circ \phi)(\omega) \quad \text{in } D_0\\
 u(\omega) &= 0 \quad \text{on } \Gamma_0,
\end{align*}
where the specific random coefficients for this particular problem are given by 
$$G(\omega) = \nabla \phi^{\top}(\omega) \nabla \phi(\omega) \quad g(\omega)= \text{det } G(\omega).$$
We now
have access to a wide breadth of numerical techniques, including Monte-Carlo \cite{cliffe2011multilevel,kuo2012quasi}
and the stochastic Galerkin method \cite{babuska2004galerkin,matthies2005galerkin}, to compute any statistical quantities of interest.  
\begin{remark}
 Note that the choice of the reference domain $D_0$, for the stochastic domain mapping $\phi$
describing the complete random geometry in question, is arbitrary and should be chosen in such a way that simplifies the computation at hand. 
 Furthermore in practice, 
only statistical properties such as the expectation and two-point covariance function of the stochastic mapping $\phi$ will be known. As a result, 
 an approximation of the true process will instead be used, commonly taking the form 
 of a truncated series 
 \[
  \phi(\omega, x) = \mathbb{E}[\phi](x) + \sum\limits_{k=1}^N Y_k(\omega) \phi_k(x) 
 \]
 with centered, uncorrelated random coefficients $Y_k$ with unit variance, such as a truncated Karhunen-Lo\`eve expansion.
 Considerations of the induced error is beyond the scope of this paper and we instead refer the reader to \cite{harbrecht2016analysis}.  
\end{remark}

\subsection{Expected domain and quantity of interest}
 In order to give a precise definition of our quantity of interest, which for our purpose shall be some notion of a mean solution, we will
  first need to fix a suitable domain of definition. A natural choice would be the parametrisation based expected domain, introduced in \cite{dambrine2017bernoulli}
  for random star-shaped domains, which we shall generalise as follows.
 \begin{definition}[Parametrisation based expected domain]
Given a family of random Lipschitz domains 
\begin{equation}\label{eq:DomainInitial}
 D(\omega) = \{ \phi(\omega,x) \, | \, x \in D_0\}\subset \re^n,
\end{equation}
parametrised over a fixed Lipchitz domain 
$D_0\subset \re^n$ under the Lipschitz continuous mapping $\phi(\omega,\cdot) : D_0 \rightarrow \re^n$. 
Assuming $\phi(\cdot,x)$ is integrable for all $x\in D_0$,
the parametrisation based expected domain $\mathbb{E}[D]$ of the random domain $D(\omega)$ is given by
\begin{equation}
\mathbb{E}[D] = \{ \mathbb{E}[\phi](x) \, | \, x \in D_0\}.
\end{equation}
 \end{definition}

 \begin{remark}
 Note that there are other alternative methods in which to define the expected value of a family of random sets. For example, we could 
  characterise the random set $D(\omega)$ as an indicator function $1_{D(\omega)}$ and then use its average, the so-called coverage function 
  $p(x) = \mathbb{P}(x \in D(\omega))$ to define the expected value to be set 
  \[
   \mathbb{E}_V[D] = \{ x \, | \, p(x) \geq \lambda\},
  \]
where the parameter $\lambda>0$ is selected in a such a way that the volume of $\mathbb{E}_V[D]$ is close as possible to the expected volume of the random sets $D(\omega)$.
This is known as the Vorob'ev expectation and was shown in \cite{dambrine2017bernoulli} not to coincide with the parameterisation based expectation.
 Although there is no canonical definition of the expected value of a random domain, the parametrisation based expected domain fits naturally in the setting of the domain
mapping method and thus will be adopted. 

 \end{remark}
 
 \begin{assumption}
 We will assume that the expected value of the stochastic mapping 
 $$\mathbb{E}[\phi]:D_0 \rightarrow \mathbb{E}[D],$$
 is bi-Lipschitz continuous. This will ensure that the parametrisation based expected domain $\mathbb{E}[D]$ 
 is also Lipschitz continuous and furthermore of the same dimension as $D_0$ and $D(\omega)$. 
 \end{assumption}
 We will denote the induced zero-mean stochastic mapping 
 between the parametrisation based expected domain $\mathbb{E}[D]$ and realisations of the random domain $D(\omega)$ by
\begin{equation}
\phi_e= \phi \circ  \mathbb{E}[\phi]^{-1}. 
\end{equation}
See figure \ref{fig:param2} for an illustration of the different mappings and domains. Our quantity of interest can now be defined on the expected domain as follows.
\begin{definition}[QoI]
 Given a random field $u(\omega,\cdot):D(\omega) \rightarrow \re$ defined over the family of random Lipschitz domains given in (\ref{eq:DomainInitial}),
 the expected value of the 
 random field is given by 
 \begin{equation}
    \text{QoI}[u] = \mathbb{E}[u \circ \phi_e] \quad  \text{on } \mathbb{E}[D].
 \end{equation}
\end{definition}
 \begin{figure}[H]
\begin{tikzpicture}
  \coordinate (d) at (0.5,0);
  \draw[black] (-4,0) circle (1cm)  node[]{\footnotesize $\mathbb{E}[D]$};
  \draw [draw=black] (-9,-0.75) rectangle (-7.5,0.75) node[pos=.5] {$D_0$};
  \draw node[] at (-4,1.3) {\footnotesize $\mathbb{E}[\Gamma]$};
  \draw node[] at (0.5,1.3) {\footnotesize $\Gamma(\omega)$};
  \draw node[] at (0.5,0) {\footnotesize $D(\omega)$};

    \draw node[] at (-8.25,1) {\footnotesize $\Gamma_0$};

    \draw[->] (-3, 0.7)  to [out=30,in=150] (-0.8,0.7);
    \draw node[] at (-1.9,1.3) {\footnotesize $\phi_e(\omega, \cdot)$};
    
    \draw[->] (-7, 0.7)  to [out=30,in=150] (-5,0.7);
    \draw node[] at (-6,1.3) {\footnotesize $\mathbb{E}[\phi](\cdot)$};

 
   \draw[black,rounded corners=1mm] (d) \irregularcircle{1cm}{.75mm};
\end{tikzpicture} 
  \caption{The computational domain, parametrisation based expected domain and a realisation of the random domain.}
  \label{fig:param2}
\end{figure}
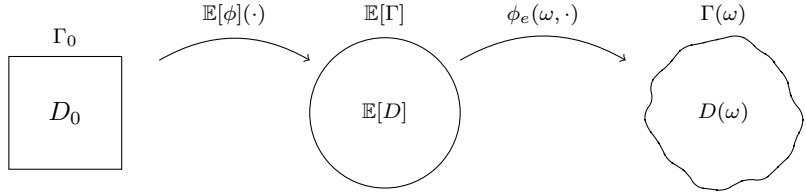
As previously discussed our aim is to apply the domain mapping method for random domains which involve random surfaces.
We will therefore now proceed with some preliminary computations of geometric quantities as well as tangential derivatives of functions 
given over parametrised hypersurfaces in terms of quantites of the reference surface and derivatives of the domain mapping and corresponding pull-back function.
This will provide a basis
for the domain mapping method to be employed to several model PDEs over random surfaces.

\section{Computations for the pull-back of tangential differential operators and geometric quanitities of parametrised hypersurfaces}
Let us first introduce some notation for hypersurfaces that will be adopted throughout this paper. For a more detailed introduction, 
see \cite{dziuk2013finite}.
\subsection{Hypersurface notation}
A set $\Gamma \subset \re^{n+1}$ is said to be a $C^k$-hypersurface for $k \in \mathbb{N} \cup \{\infty\}$, provided that for every $x \in \Gamma$ 
there exists an open set $U \subset \re^{n+1}$ containing $x$ and a smooth function $\varphi \in C^k(U)$ such that $\nabla \varphi(x) \neq 0$ on $U \cap \Gamma$ and 
\begin{equation}
U \cap \Gamma = \{x \in U \, | \, \varphi(x) =0 \}.
\end{equation}
The unit normal vector field $\nu^{\Gamma}$ to the hypersurface
$\Gamma$ can be computed via
\begin{equation}\label{eq:UnitNormalFormula}
 \nu^{\Gamma} = \pm\frac{\nabla \varphi}{|\nabla \varphi|}
\end{equation}
with a choice of orientation. For a differentiable function $f:\Gamma \rightarrow \re$, we define the tangential gradient by
\begin{equation}
\nabla_{\Gamma} f = \nabla \bar{f} - ( \nabla \bar{f}\cdot \nu^{\Gamma}) \nu^{\Gamma} = \mathcal{P}_{\Gamma} \nabla \bar{f}
\end{equation}
where $\mathcal{P}_{\Gamma} = I - \nu^{\Gamma} \otimes \nu^{\Gamma}$ is the projection operator mapping onto the tangent space $T \Gamma$ to the hypersurface $\Gamma$  
and $\bar{f}$ is a smooth extension of $f$ to an open neighbourhood in $\re^{n+1}$. 
It can be shown that the tangential gradient is independent of the extension chosen \cite[Lemma 2.4]{dziuk2013finite} and we shall denote its components by 
\[
\nabla_{\Gamma}f = (\ud_1^{\Gamma}f,..., \ud_{n+1}^{\Gamma}f)^{\top}.
\]
We shall further denote the tangential derivative of the unit normal by $\mathcal{H}^{\Gamma}= \nabla_{\Gamma}\nu^{\Gamma}$ and 
will refer to this matrix as the extended Weingarten map. It can be shown that $\mathcal{H}^{\Gamma}$ is symmetric with a zero eigenvalue 
corresponding to the unit normal vector $\nu^{\Gamma}$
and furthermore agrees with the Weingarten map when restricted to the tangent space $T\Gamma$, see \cite{deckelnick2005computation} for details.
The Laplace-Beltrami operator is then defined for twice differentiable functions as follows
\begin{equation}
\Delta_{\Gamma}f = \nabla_{\Gamma} \cdot \nabla_{\Gamma}f = \sum \limits_{i=1}^{n+1} \ud_i^{\Gamma}\ud_i^{\Gamma}f.
\end{equation}
We next introduce the Fermi coordinates with the following well-known lemma \cite[Lemma 2.8]{dziuk2013finite}. These are a global coordinate system defined in an open neighbourhood around $\Gamma$ in which every point can be uniquely expressed in terms of its signed distance $d^{\Gamma}(x)$ and its closest point $a^{\Gamma}(x)$ on the surface $\Gamma$ .
\begin{lemma}\label{lemma:fermi}
Let $d^{\Gamma}$ denote the signed distance function to $\Gamma$ oriented in the chosen direction of the unit normal vector field $\nu^{\Gamma}$. Then there exists $\delta>0$ 
such that for every $x \in U_{\delta}:= \{ y \in \re^{n+1} \, | \, |d^{\Gamma}(y)| < \delta \}$ there exists a unique point $a^{\Gamma}(x)\in \Gamma$ that satisfies
\begin{equation}
\label{eq:aProj}
x = a^{\Gamma}(x)  +d^{\Gamma}(x) \nu^{\Gamma}(a^{\Gamma}(x)).
\end{equation}
Furthermore, assuming $\Gamma \in C^2$ it follows that $d^{\Gamma}\in C^2(U_{\delta})$ and $a^{\Gamma}\in C^{1}(U_{\delta})$ with 
\begin{align}
    \nabla d^{\Gamma}(x) &= \nu^{\Gamma}(a^{\Gamma}(x)) \label{fermi:d}\\
    \nabla a^{\Gamma}(x) &= \left(I+ d^{\Gamma}(x) \mathcal{H}^{\Gamma}(a^{\Gamma}(x))\right)^{-1} \mathcal{P}_{\Gamma}(a^{\Gamma}(x)) \label{fermi:a}.
\end{align}
\end{lemma}

\subsection{Geometric settings}
As a point of reference, we will now describe the deterministic geometric settings that will be considered
for the parametrised surfaces in the subsequent calculations.
In each case, the reference surface $\Gamma_0\subset \re^{n+1}$ will be assumed to be of class at least $C^2$ and oriented by the unit normal vector field $\nu^{\Gamma_0}$.
The general geometric setting for the parametrised surface will be as follows. 
\begin{GeoSet}[Parametrised surface]
The hypersurface $\Gamma \subset \re^{n+1}$ will be given by
\begin{equation}\label{Geometric:setting1}
 \Gamma = \{ \phi(x) \, | \, x \in \Gamma_0\},
\end{equation}
for a given mapping $\phi:\Gamma_0 \rightarrow \re^{n+1}$.
\end{GeoSet}
We will further consider the special case, where the parametrised surface has the following graph-like representation over the reference surface.
\begin{GeoSet}[Graph-like surface]
The surface $\Gamma\subset \re^{n+1}$ will be prescribed by
 \begin{equation}\label{Geometric:setting2}
 \Gamma = \{ \phi(x) = x + h(x) \nu^{\Gamma_0}(x) \, | \, x \in \Gamma_0\} 
\end{equation}
for a given height function $h:\Gamma_0 \rightarrow \re$ defined over the reference surface.
\end{GeoSet}
Additionally, we will consider the case where the surface is compact (and thus without a boundary) and encloses an open bulk domain.
\begin{GeoSet}[Parametrised bulk-surface]
The open bulk domain $D\subset \re^{n+1}$ and its boundary $\Gamma = \partial D$ which is a surface, will be given by 
 \begin{equation}\label{Geometric:setting3}
 D = \{ \phi(x) \, | \, x \in D_0\} \qquad
 \Gamma = \{ \phi(x) \, | \, x \in \Gamma_0\}
\end{equation}
for a given parametrisation $\phi:\overline{D_0} \rightarrow \re^{n+1}$ defined over an 
open bulk domain $D_0\subset \re^{n+1}$ with boundary $\Gamma_0 = \partial D_0$.
\end{GeoSet}
Note that in each case, the given parametrisation $\phi$ will be assumed to be a sufficiently smooth diffeomorphism for the calculation in question.
 Furthermore, we shall denote the associated pull-back of a given function $f$ defined over the parametrised domain onto the reference domain by
\begin{equation}\label{eq:pullBackNotation}
\hat{f} = f \circ \phi.
\end{equation}

\subsection{The tangential gradient and Laplace-Beltrami operator}
Considering a general parametrised hypersurface $\Gamma$ as described in (\ref{Geometric:setting1}), 
we will now compute expressions for the pull-back of the
 tangential gradient $\nabla_{\Gamma}$ and Laplace-Beltrami operator $\Delta_{\Gamma}$ onto the reference surface $\Gamma_0$ under the domain mapping $\phi$.
 As a motivation for these calculations, 
 let us first recall that for a given local parametrisation
 \[
  X: U \rightarrow W \cap \Gamma
 \]
 of the hypersurface $\Gamma$, where $U\subset \re^n$ and $W\subset \re^{n+1}$ denote open sets,
 we can express the tangential gradient and Laplace-Beltrami operator in local coordinates as follows
 \begin{align}
 \label{eq:surfGrad_local}
  \nabla_{\Gamma} f \circ X &= \nabla X G^{-1} \nabla F\\
   \label{eq:LB_local}
  \Delta_{\Gamma}f\circ X &= \frac{1}{\sqrt{g}} \nabla \cdot \left( \sqrt{g} G^{-1} \nabla F \right)
 \end{align}
where $F=f\circ X$ and the first fundamental form $G:U \rightarrow \re^{n\times n}$ is defined as 
$G= \nabla X^{\top} \nabla X$
with $g=\text{det} G$.
In deriving expressions for the pull-back onto the reference surface $\Gamma_0$ instead of the local coordinates, 
we will see similar expresssions to (\ref{eq:surfGrad_local}) and (\ref{eq:LB_local}) but with the first fundamental 
form replaced by the following tensor $G_{\Gamma_0}:\Gamma_0 \rightarrow \re^{(n+1)\times (n+1)}$ defined by 
\begin{equation}
\label{eq:G_eq}
G_{\Gamma_0} = \nabla_{\Gamma_0} \phi^{\top} \nabla_{\Gamma_0} \phi + \nu^{\Gamma_0} \otimes \nu^{\Gamma_0},
\end{equation}
where we will similarly denote its determinant by 
$g_{\Gamma_0} = \text{det} G_{\Gamma_0}$.
This tensor can be seen to arise by considering a local parametrisation 
$\sigma: U \rightarrow V\cap \Gamma_0$
of the reference surface $\Gamma_0$, with $V\subset \re^{n+1}$ denoting an open set, and the induced local parametrisation 
$$
X = \phi \circ \sigma: U \rightarrow W\cap \Gamma
$$
of the hypersurface $\Gamma$.
By computing the first fundamental form with the chain rule, we observe that
\begin{align*}
 G &= \nabla \sigma^{\top} \left(\nabla_{\Gamma_0}\phi^{\top} \circ \sigma \right) \left( \nabla_{\Gamma_0} \phi \circ \sigma  \right)\nabla \sigma.
\end{align*}
Since $ \nabla_{\Gamma_0}\phi^{\top}\nabla_{\Gamma_0}\phi \nu^{\Gamma_0} =0$ and its restriction to the tangent space maps 
$\nabla_{\Gamma_0}\phi^{\top}\nabla_{\Gamma_0}\phi: T\Gamma_0 \rightarrow T \Gamma_0$, we are able to extend in the normal direction as in (\ref{eq:G_eq}) 
to form an invertible matrix. Furthermore as  $\nabla \sigma \in T\Gamma_0$, it follows that we have 
 \begin{align*}
 G  &= \nabla \sigma^{\top} \left(\nabla_{\Gamma_0}\phi^{\top} \nabla_{\Gamma_0} \phi + \nu^{\Gamma_0}\otimes \nu^{\Gamma_0} \right)\circ \sigma \, \nabla \sigma
= \nabla \sigma^{\top} \left(G_{\Gamma_0} \circ \sigma\right) \nabla \sigma.
\end{align*}
Note that the given extension (\ref{eq:G_eq}) in the normal direction is a natural choice since the surface measures $dA_{\Gamma}$ and $dA_{\Gamma_0}$
of the respective surfaces 
can be shown to satisfy the relation
$dA_{\Gamma} = \sqrt{g_{\Gamma_0}} dA_{\Gamma_0}$ under the domain transformation mapping $\phi$.
We now continue by proving that a similar expression to (\ref{eq:surfGrad_local}) holds for the pull-back of the tangential gradient.
\begin{lemma}[Tangential gradient]\label{theorem:surfGrad}
 Given any differentiable function $f:\Gamma \rightarrow \re$, the pull-back of the tangential gradient onto the reference surface $\Gamma_0$ is given by 
\begin{align}\label{eq:Surfgrad} 
    \left( \nabla_{\Gamma} f\right) \circ \phi &= \left( \nabla_{\Gamma_0} \phi
    + \nu^{\Gamma} \circ \phi \otimes \nu^{\Gamma_0} \right)^{-\top} \nabla_{\Gamma_0} \hat{f}
    = \nabla_{\Gamma_0}\phi \, G_{\Gamma_0}^{-1} \nabla_{\Gamma_0}\hat{f}.
\end{align}
\end{lemma}
 
\begin{proof}
Differentiating the associated pull-back function $\hat{f} = f\circ \phi$ and applying the chain rule for 
tangential derivatives gives
\begin{equation}
 \label{eq:ProofSurfaceGrad}
 \nabla_{\Gamma_0} \hat{f} =  \nabla_{\Gamma_0}\phi^{\top} \left(\nabla_{\Gamma}f \right)\circ \phi.
\end{equation}
Since the tangential gradient of the surface parametrisation bijectively maps 
$\nabla_{\Gamma_0}\phi: T_{(\cdot)}\Gamma_0 \rightarrow T_{\phi(\cdot)}\Gamma$ and additionally
has kernel equal to span$\{\nu^{\Gamma_0}\}$, we see that in order to invert the matrix $\nabla_{\Gamma_0}\phi$, we must 
 first modify the corresponding linear map to bijectively map the space span$\{\nu^{\Gamma_0}\}$ 
into span$\{\nu^{\Gamma} \circ \phi\}$. 
 One possible solution is to add the linear map $L: \re^{n+1}\rightarrow \re^{n+1}$ characterised by  
\[
L(\nu^{\Gamma_0}) = \nu^{\Gamma}\circ \phi, \quad L(\tau) = 0 \quad \tau \in T\Gamma_0,
\]
which translates to adding the following tensor product 
\begin{equation}\label{indep:Exten}
\nabla_{\Gamma_0} \hat{f} = \nabla_{\Gamma_0}\phi^{\top} \left(\nabla_{\Gamma} f \right) \circ\phi 
=  \left( \nabla_{\Gamma_0} \phi + \nu^{\Gamma}\circ \phi \otimes \nu^{\Gamma_0}\right)^{\top} \left(\nabla_{\Gamma} f\right)\circ \phi.
\end{equation}
and thus leads to (\ref{eq:Surfgrad}). For the second equality, we again use the property that the restriction
$\nabla_{\Gamma_0}\phi: T_{(\cdot)}\Gamma_0 \rightarrow T_{\phi(\cdot)}\Gamma$ is a bijective mapping
to express 
$
 (\nabla_{\Gamma} f ) \circ \phi = \nabla_{\Gamma_0}\phi \, \alpha
$
for some $\alpha\in T\Gamma_0$. Substituting into (\ref{eq:ProofSurfaceGrad}) then gives
$$
\nabla_{\Gamma_0}\hat{f} = \nabla_{\Gamma_0}\phi^{\top} \nabla_{\Gamma_0}\phi \alpha 
= \left(\nabla_{\Gamma_0}\phi^{\top} \nabla_{\Gamma_0}\phi+ \nu^{\Gamma_0} \otimes \nu^{\Gamma_0}\right) \alpha.
$$
Hence we deduce 
$\alpha = G_{\Gamma_0}^{-1} \nabla_{\Gamma_0}\hat{f}$
and obtain
the second equality.
\end{proof}

\begin{remark}
Note that the chain rule for tangential gradients (\ref{eq:Surfgrad}) holds for any choice of orientation of the unit normals $\nu^{\Gamma_0}$ and $\nu^{\Gamma}$ 
as a result of (\ref{indep:Exten}).
\end{remark}
Let us denote the given extension of the tangential gradient of the surface parametrisation appearing in (\ref{eq:Surfgrad}) by $B= \left(b_{ij}\right)_{i,j}$,
\begin{equation}
\label{eq:Bdef}
B = \nabla_{\Gamma_0} \phi + \nu^{\Gamma}\circ \phi \otimes \nu^{\Gamma_0}
\end{equation}
and furthermore denote its determinant by $b= \det B$ and the entries of its inverse by $B^{-1}=\left(b^{ij} \right)_{i,j} $. We observe 
with the orthogonality result $\nabla_{\Gamma_0}\phi^{\top} ( \nu^{\Gamma} \circ \phi) = 0$ which follows from the property that the restriction maps 
$\nabla_{\Gamma_0}\phi: T_{(\cdot)}\Gamma_0 \rightarrow T_{\phi(\cdot)}\Gamma$, that
\begin{align}
\label{eq:BandG}
B^{\top}B &= \left( \nabla_{\Gamma_0}\phi^{\top} + \nu^{\Gamma_0} \otimes (\nu^{\Gamma}\circ\phi) \right) \left( \nabla_{\Gamma_0} \phi 
+ (\nu^{\Gamma}\circ\phi) \otimes \nu^{\Gamma_0} \right)= \nabla_{\Gamma_0}\phi^{\top} \nabla_{\Gamma_0}\phi + \nu^{\Gamma_0}\otimes \nu^{\Gamma_0}  = G_{\Gamma_0}.
\end{align}
Consequently, we have
\begin{equation}
 \label{eq:bandg}
 b = \det (B) = \sqrt{\det (B^{\top} B)} = \sqrt{\det G_{\Gamma_0}} = \sqrt{g_{\Gamma_0}}.
\end{equation}
We can now compute the pull-back of the Laplace-Beltrami operator onto the reference surface as follows.
\begin{lemma}[Laplace-Beltrami operator]\label{theorem:LaplaceBelt}
 Given any $f:\Gamma \rightarrow \re$ twice differentiable, the pull-back of the Laplace-Beltrami operator is given by
\begin{align}
    \label{eq:Laplacebeltrami} 
        (\Delta_{\Gamma}f)\circ \phi &= \frac{1}{\sqrt{g_{\Gamma_0}}} \nabla_{\Gamma_0}\cdot 
        \left( \sqrt{g_{\Gamma_0}} G_{\Gamma_0}^{-1} \nabla_{\Gamma_0} \hat{f} \right).
\end{align}
\end{lemma}

\begin{proof}
By the chain rule for tangential gradients (\ref{eq:Surfgrad}), we can express the Laplace-Beltrami operator as
\begin{align*}
(\Delta_{\Gamma}f) \circ \phi &= \sum\limits_{i=1}^{n+1} (\ud_i^{\Gamma}\ud_i^{\Gamma}f ) \circ \phi
= \sum\limits_{i,j=1}^{n+1} b^{ji} \ud_j^{\Gamma_0} \left(\ud_i^{\Gamma}f \circ \phi\right)
=\sum\limits_{i,j,k=1}^{n+1}b^{ji}\ud_j^{\Gamma_0} \left( b^{ki} \ud_k^{\Gamma_0} \hat{f} \right).
\intertext{Writing in divergence form gives}
(\Delta_{\Gamma}f) \circ \phi&= \sum\limits_{i,j,k=1}^{n+1} \frac{1}{b} \ud_j^{\Gamma_0} \left( b b^{ji} b^{ki} \ud_k^{\Gamma_0} \hat{f} \right)
- \sum\limits_{i,j,k=1}^{n+1} \frac{1}{b} \ud_j^{\Gamma_0}b\, b^{ji} b^{ki} \ud_k^{\Gamma_0}\hat{f}
- \sum\limits_{i,j,k=1}^{n+1} \ud_j^{\Gamma_0}b^{ji} \,b^{ki} \ud_k^{\Gamma_0}\hat{f}\\
&= \frac{1}{b} \nabla_{\Gamma_0} \cdot \left( b B^{-1} B^{-\top} \nabla_{\Gamma_0}\hat{f}\right) +\RN{1}+\RN{2}\\
&=\frac{1}{\sqrt{g_{\Gamma_0}}} \nabla_{\Gamma_0}\cdot \left( \sqrt{g_{\Gamma_0}} G_{\Gamma_0}^{-1} \nabla_{\Gamma_0}\hat{f} \right) +\RN{1}+\RN{2}.
\end{align*}
The last step follows from the observations (\ref{eq:BandG}) and (\ref{eq:bandg}).
We continue by proving that the remaining terms vanish. Recalling Jacobi's formula 
for the derivative of a determinant $\ud_j^{\Gamma_0} \text{det} B = \text{det} B\, \text{trace} \left(B^{-1} \ud_j^{\Gamma_0} B \right)$ 
and computing the derivative of the inverse matrix $\ud_j^{\Gamma_0} B^{-1} = - B^{-1} \ud_j^{\Gamma_0}B B^{-1}$ gives
$$
\frac{1}{b} \ud_j^{\Gamma_0} b = \sum\limits_{l,m=1}^{n+1} b^{lm}\ud_j^{\Gamma_0}b_{ml}, \qquad \ud_j^{\Gamma_0}b^{ji} = - \sum\limits_{l,m=1}^{n+1} b^{jm} \ud_{j}^{\Gamma_0}b_{ml} b^{li}.
$$
 It therefore follows after relabelling indices that
\begin{align}
\RN{1} + \RN{2}&= - \sum_{i,j,k,l,m} b^{lm} \ud_j^{\Gamma_0}b_{ml} b^{ji} b^{ki}\ud_k^{\Gamma_0} \hat{f} \nonumber
+ \sum_{i,j,k,l,m} b^{jm}\ud_j^{\Gamma_0} b_{ml} b^{li} b^{ki}\ud_k^{\Gamma_0} \hat{f}\\
&= \sum_{i,j,k,l,m} b^{lm} \left( \ud_{l}^{\Gamma_0} b_{mj} - \ud_j^{\Gamma_0} b_{ml} \right) b^{ji} b^{ki} \ud_k^{\Gamma_0}\hat{f}\label{I:II}
\end{align}
Differentiating $b_{mj}:= \ud_j^{\Gamma_0}\phi_m + (\nu_m^{\Gamma}\circ \phi) \nu_j^{\Gamma_0}$ yields
\begin{align*}
    \ud_l^{\Gamma_0}b_{mj} - \ud_{j}^{\Gamma_0} b_{ml} 
    &= \ud_l^{\Gamma_0}\ud_j^{\Gamma_0}\phi_m - \ud_j^{\Gamma_0}\ud_l^{\Gamma_0}\phi_m
    + \ud_l^{\Gamma_0}(\nu_m^{\Gamma}\circ \phi) \nu_j^{\Gamma_0} - \ud_j^{\Gamma_0}(\nu_m^{\Gamma}\circ \phi) \nu_l^{\Gamma_0}\\
    &+ (\nu_m^{\Gamma}\circ \phi) \ud_l^{\Gamma_0} \nu_j^{\Gamma_0} - (\nu_m^{\Gamma}\circ \phi) \ud_j^{\Gamma_0}\nu_l^{\Gamma_0}.
\end{align*}
By the symmetry of the Weingarten map $\ud_l^{\Gamma_0}\nu_j^{\Gamma_0} = \ud_j^{\Gamma_0}\nu_l^{\Gamma_0}$, we see that the last two terms cancel. 
We next interchange tangential derivatives \cite[Lemma 2.6]{dziuk2013finite}   
$$
\ud_l^{\Gamma_0}\ud_j^{\Gamma_0} \phi_m - \ud_j^{\Gamma_0}\ud_l^{\Gamma_0} \phi_m
= \left(\mathcal{H}^{\Gamma_0}\nabla_{\Gamma_0} \phi_m \right)_j \nu_l^{\Gamma_0} - 
\left(\mathcal{H}^{\Gamma_0}\nabla_{\Gamma_0} \phi_m \right)_l \nu_j^{\Gamma_0}
$$
to obtain
\begin{align*}
    \ud_l^{\Gamma_0}b_{mj} - \ud_{j}^{\Gamma_0} b_{ml} 
    &= \left( \ud_l^{\Gamma_0}(\nu_m^{\Gamma}\circ \phi) - ( \mathcal{H}^{\Gamma_0}\nabla_{\Gamma_0}\phi_m)_l \right) \nu_j^{\Gamma_0}
    +\left( (\mathcal{H}^{\Gamma_0}\nabla_{\Gamma_0}\phi_m)_j - \ud_j^{\Gamma_0}(\nu_m^{\Gamma}\circ \phi)  \right) \nu_l^{\Gamma_0}.
\end{align*}
Substituting into (\ref{I:II}), we arrive at the following expression for the remaining terms 
\begin{align*}
\RN{1} + \RN{2} &= trace \left( B^{-1}\nabla_{\Gamma_0}(\nu^{\Gamma}\circ \phi) -  B^{-1}\nabla_{\Gamma_0}\phi \mathcal{H}^{\Gamma_0} \right) B^{-\top} \nu^{\Gamma_0} \cdot B^{-\top} \nabla_{\Gamma_0} \hat{f} \\
&+  \mathcal{H}^{\Gamma_0}\nabla_{\Gamma_0}\phi^{\top}  B^{-\top}\nu^{\Gamma_0} \cdot B^{-1}B^{-\top}\nabla_{\Gamma_0}\hat{f}\\
&- \nabla_{\Gamma_0}(\nu^{\Gamma}\circ \phi)^{\top}B^{-\top}\nu^{\Gamma_0} \cdot B^{-1}B^{-\top}\nabla_{\Gamma_0}\hat{f}.
\end{align*}
Examining the first term, we have
\[
B^{-\top}\nu^{\Gamma_0} \cdot B^{-\top} \nabla_{\Gamma_0} \hat{f} 
= B^{-1}B^{-\top} \nu^{\Gamma_0} \cdot \nabla_{\Gamma_0} \hat{f} = G_{\Gamma_0}^{-1} \nu^{\Gamma_0}\cdot \nabla_{\Gamma_0}\hat{f}.
\]
Since $G_{\Gamma_0}= \nabla_{\Gamma_0}\phi^{\top}\nabla_{\Gamma_0} \phi + \nu^{\Gamma_0}\otimes \nu^{\Gamma_0}$
and thus $G_{\Gamma_0}^{-1}\nu^{\Gamma_0} = \nu^{\Gamma_0}$, the first term vanishes.
For the second and third term, we observe that
\[
B^{-\top}\nu^{\Gamma_0} = B B^{-1}B^{-\top} \nu^{\Gamma_0} = B G_{\Gamma_0}^{-1}\nu^{\Gamma_0} = B \nu^{\Gamma_0} = \nu^{\Gamma}\circ \phi.
\]
Therefore as a consequence of the orthogonality results $\nabla_{\Gamma_0}\phi^{\top} (\nu^{\Gamma}\circ \phi) = 0$ and 
$\nabla_{\Gamma_0}(\nu^{\Gamma}\circ \phi)^{\top}(\nu^{\Gamma}\circ \phi) = 0$ which can be seen by
$$
\ud_i^{\Gamma_0}(\nu^{\Gamma}\circ \phi) \cdot (\nu^{\Gamma}\circ \phi) = \frac{1}{2}\ud_i^{\Gamma}|\nu^{\Gamma}\circ \phi|^2 = 0,
$$
we conclude $\RN{1} + \RN{2} = 0$.
\end{proof}
We next compute the specific form of the coefficients appearing in the pull-back of the tangential gradient (\ref{eq:Surfgrad}) and 
the Laplace-Beltrami operator (\ref{eq:Laplacebeltrami}), for the particular case of a graph-like parametrisation over the reference surface.
\begin{lemma}[Graphical case]
\label{lemma:graphCoef}
Assuming that the parametrisation of the hypersurface $\Gamma$
has the particular graph-like representation described in (\ref{Geometric:setting2}) for a given height function
$h: \Gamma_0\rightarrow \re$, then the inverse and determinant of the tensor $G_{\Gamma_0}$ defined in (\ref{eq:G_eq}) simplify to give
\begin{align}
\label{eq:Ggraph}
G_{\Gamma_0}^{-1} &= A  \left( I - \frac{A  \nabla_{\Gamma_0}h  \otimes A  \nabla_{\Gamma_0}h }
{1 + |A  \nabla_{\Gamma_0}h |^2}\right) A \\
\label{eq:ggraph}
\sqrt{g_{\Gamma_0}}  &=\sqrt{1+ |A  \nabla_{\Gamma_0}h |^2} \prod_{j=1}^n ( 1 +h  \kappa_j^{\Gamma_0}).
\end{align}
Here $A:= \left( I +h \mathcal{H}^{\Gamma_0}\right)^{-1}$ and $\{\kappa_j^{\Gamma_0}\}_j$ denotes the non-zero eigenvalues of the extended Weingarten map $\mathcal{H}^{\Gamma_0}$.
\end{lemma}

\begin{proof}
Differentiating the given surface parametrisation $\phi(x) = x + h(x) \nu^{\Gamma_0}(x)$, we obtain
\begin{equation*}
\nabla_{\Gamma_0}\phi = \mathcal{P}_{\Gamma_0} + h \mathcal{H}^{\Gamma_0} + \nu^{\Gamma_0}\otimes \nabla_{\Gamma_0}h.
\end{equation*}
Expanding the tensor $G_{\Gamma_0} = \nabla_{\Gamma_0}\phi^{\top}\nabla_{\Gamma_0}\phi + \nu^{\Gamma_0}\otimes \nu^{\Gamma_0}$
and cancelling the orthogonal terms with the tensor product identity $(a\otimes b)(c\otimes d) = (b\cdot c)a\otimes d$, 
yields
\begin{align*}
G_{\Gamma_0}  &= \left(
\mathcal{P}_{\Gamma_0} + h \mathcal{H}^{\Gamma_0} +\nabla_{\Gamma_0}h \otimes \nu^{\Gamma_0}
\right)
\left(
\mathcal{P}_{\Gamma_0} + h \mathcal{H}^{\Gamma_0} + \nu^{\Gamma_0}\otimes \nabla_{\Gamma_0}h
\right)
+ \nu^{\Gamma_0} \otimes \nu^{\Gamma_0}
\\
&=\left( I + h \mathcal{H}^{\Gamma_0}\right)^2 + \nabla_{\Gamma_0}h \otimes \nabla_{\Gamma_0}h\\
&= A^{-1}\left( I + A \nabla_{\Gamma_0}h \otimes A \nabla_{\Gamma_0}h \right) A^{-1}.
\end{align*}
Taking the inverse with the identity $\left(I + a \otimes b\right)^{-1}= I - \frac{a\otimes b}{1 + a \cdot b}$ we obtain (\ref{eq:Ggraph}). 
For (\ref{eq:ggraph}), we take the determinant and apply $\det(I+ a\otimes b) = 1 + a \cdot b$, which leads to
\begin{equation*}
\det(G_{\Gamma_0} ) = \left(1 + |A \nabla_{\Gamma_0}h|^2 \right)\det(A^{-1})^2.
\end{equation*}
Since $A^{-1}= I + h \mathcal{H}^{\Gamma_0}$ has eigenvalues $1$ and $\{1+ h \kappa_j^{\Gamma_0} \}_{j=1}^n$, we deduce
$
\det(A^{-1}) = \prod_{j=1}^n\left(1 + h \kappa_j^{\Gamma_0}\right)
$
and thus obtain the stated result for $\sqrt{g_{\Gamma_0}} = \sqrt{ det G_{\Gamma_0}}$.
\end{proof}

\subsection{The unit normal and extended Weingarten map}
We continue by computing expressions for the pull-back onto the reference surface $\Gamma_0$, of the unit normal $\nu^{\Gamma}$ and extended Weingarten map
$\mathcal{H}^{\Gamma}$ for a general parametrised hypersurface $\Gamma$ as given in (\ref{Geometric:setting1}). 
To obtain an expression for the unit normal, we smoothly extend
the given surface parametrisation $\phi: \Gamma_0 \rightarrow \Gamma$ to a $C^1-$diffeomorphic mapping
$\bar{\phi}: U \rightarrow V$ 
between some open sets $U$ and $V$ containing $\Gamma_0$ and $\Gamma$ respectively. 
The existence of such a mapping is gauranteed by the Whitney extension theorem \cite{whitney1934analytic}. 
We now have a level-set description of $\Gamma$
\[
 \Gamma = \{ x \in V \, | \, d^{\Gamma_0}( \bar{\phi}^{-1}(x) ) = 0 \}
\]
consequently leading to the following expression for the unit normal vector field due to (\ref{eq:UnitNormalFormula}).
\begin{lemma}[Unit normal]\label{theorem:unitNormal}
 The pull-back of the unit normal vector field $\nu^{\Gamma}$ of the parametrised surface $\Gamma$ described in (\ref{Geometric:setting1}) onto the reference surface $\Gamma_0$,
 is given by 
 \begin{equation}\label{eq:UnitNormal}
  \nu^{\Gamma} \circ \phi = \pm \frac{\nabla \bar{\phi}^{-\top} \nu^{\Gamma_0}}{|\nabla \bar{\phi}^{-\top} \nu^{\Gamma_0}|}.
 \end{equation}

\end{lemma}
Note that (\ref{eq:UnitNormal}) can be shown to be independent of the extension chosen. 
As an example of a possible extension of the given surface parametrisation, we now consider the case of a graph-like surface.

\begin{corollary}[Graphical case]
Assuming that the hypersurface $\Gamma$ has the particular graph-like form described in (\ref{Geometric:setting2}), then the pull-back of the unit 
normal vector field $\nu^{\Gamma}$ is given by
 \begin{equation}
  \nu^{\Gamma}\circ \phi = \frac{\nu^{\Gamma_0}- A \nabla_{\Gamma_0}h }{|\nu^{\Gamma_0}- A \nabla_{\Gamma_0}h |}.
 \end{equation}
 Here the orientation has been chosen to coincide with the reference surface $\Gamma_0$ when the height function is identically zero.
 Recall that $A := (I + h \mathcal{H}^{\Gamma_0})^{-1}$.  
\end{corollary}

\begin{proof}
We extend the given surface parametrisation $\phi:\Gamma_0\rightarrow \Gamma$ defined by 
$$
\phi(x) = x + h(x)\nu^{\Gamma_0}(x),
$$
to a thin tubular neighbourhood 
$
U = \{ x \in \re^{n+1} \, | \, |d^{\Gamma_0}(x)| < \delta \}
$
around $\Gamma_0$ of width $\delta>0$ as follows 
\begin{align*}
 \bar{\phi}(x) 
 &= \phi(a^{\Gamma_0}(x)) + d^{\Gamma_0}(x) \nu^{\Gamma_0}(a^{\Gamma_0}(x)) \\
 &= a^{\Gamma_0}(x) + \left(h(a^{\Gamma_0}(x)) +  d^{\Gamma_0}(x) \right) \nu^{\Gamma_0}(a^{\Gamma_0}(x)). 
\end{align*}
For $\delta>0$ sufficiently small, its image $V= \bar{\phi}(U)$ is contained within the neighbourhood in which the Fermi coordinates
$(a^{\Gamma_0}(x), d^{\Gamma_0}(x))$ are well defined. Consequently, the extension $\bar{\phi}:U\rightarrow V$ which equivalently acts upon the Fermi coordinates
as follows
\[
 (a^{\Gamma_0}(x), d^{\Gamma_0}(x)) \mapsto  \left(a^{\Gamma_0}(x), d^{\Gamma_0}(x)+ h(a^{\Gamma_0}(x)) \right) 
\]
can be seen to be a bijective mapping. 
Computing its derivative and evaluating on the reference surface $\Gamma_0$, recalling that $\nabla d^{\Gamma_0} = \nu^{\Gamma_0}$ and 
$\nabla a^{\Gamma_0} = \mathcal{P}_{\Gamma_0}$ on $\Gamma_0$ by (\ref{fermi:d}) and (\ref{fermi:a}), we obtain
\[
 \nabla \bar{\phi} = I + h \mathcal{H}^{\Gamma_0} + \nu^{\Gamma_0} \otimes \nabla_{\Gamma_0}h = \left( I + \nu^{\Gamma_0} \otimes A \nabla_{\Gamma_0}h \right) A^{-1}.
\]
Hence taking the inverse with the identity $( I + a\otimes b) ^{-1} = 1 - \frac{a\otimes b}{1 + a \cdot b}$ and recalling that $A\nu^{\Gamma_0} = \nu^{\Gamma_0}$, 
we deduce
\[
 \nabla \bar{\phi}^{-\top} \nu^{\Gamma_0} = \left( I - \nu^{\Gamma_0} \otimes A \nabla_{\Gamma_0}h \right) A \nu^{\Gamma_0} = \nu^{\Gamma_0} - A \nabla_{\Gamma_0}h
\]
and thus obtain the stated result. Note that $\nu^{\Gamma_0} - A \nabla_{\Gamma_0}h \neq 0$ 
since the matrix $A = (I + h \mathcal{H}^{\Gamma_0})^{-1}$ maps $A:T\Gamma_0 \rightarrow T\Gamma_0$.
\end{proof}

We next compute the pull-back of the extended Weingarten map $\mathcal{H}^{\Gamma}$ for a general parametrised surface $\Gamma$. 
Since the restriction of the derivative of the surface parametrisation maps 
$\nabla_{\Gamma_0}\phi(\cdot):T_{(\cdot)}\Gamma_0\rightarrow T_{\phi(\cdot)}\Gamma$, we consequently have 
\[
 \left(\nu^{\Gamma}\circ \phi\right) \cdot \ud_j^{\Gamma_0}\phi =0
\]
 for all $j=1,...,n+1$. Differentiating, we obtain
\begin{equation}
\label{eq:WeingPrelimEq}
 \ud_i^{\Gamma_0}\left( \nu^{\Gamma}\circ \phi \right) \cdot \ud_j^{\Gamma_0}\phi = - \left(\nu^{\Gamma}\circ \phi\right) \cdot \ud_i^{\Gamma_0}\ud_j^{\Gamma_0}\phi.
\end{equation}
Next, we define $L:\Gamma_0 \rightarrow \re^{(n+1)\times(n+1)}$ by
\begin{equation}
\label{eq:L}
 \left(L(x) \right)_{i,j} = \left(\nu^{\Gamma}\circ \phi\right)(x) \cdot \ud_i^{\Gamma_0}\ud_j^{\Gamma_0}\phi(x) \qquad x \in \Gamma_0,
\end{equation}
and rewrite (\ref{eq:WeingPrelimEq}) as  
\[
\nabla_{\Gamma_0}\left(\nu^{\Gamma}\circ \phi \right)^{\top} \nabla_{\Gamma_0}\phi = - L. 
\]
It therefore follows from an application of the chain rule and the symmetry of the extended Weingarten map that
\[
  \nabla_{\Gamma_0}\phi^{\top}\left(\mathcal{H}^{\Gamma}\circ \phi \right)\nabla_{\Gamma_0} \phi
  = \nabla_{\Gamma_0}\phi^{\top}\left( \nabla_{\Gamma} \nu^{\Gamma} \circ \phi \right)^{\top}\nabla_{\Gamma_0} \phi = - L.
\]
We can then extend the tangential derivative $\nabla_{\Gamma_0}\phi$ to an invertible matrix as previously discussed in (\ref{eq:Bdef}), to obtain 
the following result.

\begin{lemma}[Extended Weingarten map]
Let the orientation of the parametrised hypersurface $\Gamma$ described in (\ref{Geometric:setting1}), be fixed by a choice of a unit normal vector field $\nu^{\Gamma}$. 
Then the pull-back of the extended Weingarten map is given by
 \begin{equation}\label{eq:Weingarten}
  \mathcal{H}^{\Gamma}\circ \phi 
  = - \left( \nabla_{\Gamma_0} \phi + \nu^{\Gamma}\circ \phi \otimes \nu^{\Gamma_0} \right)^{-\top}
  L
  \left( \nabla_{\Gamma_0} \phi + \nu^{\Gamma}\circ \phi \otimes \nu^{\Gamma_0} \right)^{-1}.
 \end{equation}
\end{lemma}
Note that the matrix $L(x)$ given in (\ref{eq:L}) is symmetric even though the tangential derivatives do not necessarily commute, as 
by interchanging the derivatives we obtain 
 \begin{align*}
 (\nu^{\Gamma}\circ \phi) \cdot \ud_i^{\Gamma_0} \ud_j^{\Gamma_0}\phi 
 = (\nu^{\Gamma}\circ \phi) \cdot \ud_j^{\Gamma_0} \ud_i^{\Gamma_0}\phi 
 + \sum\limits_{m=1}^{n+1} \mathcal{H}_{jm}^{\Gamma_0} \left( \ud_m^{\Gamma_0} \phi \cdot\left( \nu^{\Gamma}\circ \phi\right) \right) \nu_i^{\Gamma_0}
 - \sum\limits_{m=1}^{n+1} \mathcal{H}_{im}^{\Gamma_0} \left( \ud_m^{\Gamma_0}\phi \cdot \left(\nu^{\Gamma}\circ \phi \right)\right) \nu_j^{\Gamma_0}.
\end{align*}
and since $\ud_m^{\Gamma_0}\phi \cdot \left(\nu^{\Gamma}\circ \phi\right) = 0$, for all $m=1,...,n+1$, we see that the last two terms vanish.  


\subsection{The normal derivative at the boundary}
We conclude this section by computing the pull-back of the normal derivative at the boundary for functions defined over the parametrised bulk-surface
described in (\ref{Geometric:setting3}).

\begin{lemma}[Normal derivative]
 Given any $u: \bar{D}\rightarrow \re$ sufficiently smooth, the pull-back of its normal derivative is given by
 \begin{equation}\label{eq:NormalDeriv}
 \frac{\partial u}{\partial \nu_{\Gamma}}\circ \phi = \frac{\sqrt{g}}{\sqrt{g_{\Gamma_0}}}
 \left( \mathcal{P}_{\Gamma_0} G^{-1} \nu^{\Gamma_0} \cdot \nabla_{\Gamma_0} \hat{u} + 
 \left( G^{-1} \nu^{\Gamma_0}\cdot \nu^{\Gamma_0} \right) \frac{\partial \hat{u} }{\partial \nu_{\Gamma_0}} \right)
 \end{equation}
 where $G = \nabla \phi^{\top} \nabla \phi$ and $g = det(G)$ denoting its determinant.
\end{lemma}

\begin{proof}
 Differentiating $u = \hat{u} \circ \phi^{-1}$ and substituting in the expression (\ref{eq:UnitNormal}) for the
 pull-back of the unit normal $\nu^{\Gamma}$, where the orientation has been chosen to be in the outer 
 direction to the domain $D$ gives
 \[
  \frac{\partial u}{\partial \nu_{\Gamma}} = \nabla u \cdot \nu^{\Gamma} 
  = \nabla \phi^{-\top} (\nabla \hat{u} \circ \phi^{-1}) \cdot \frac{\nabla \phi^{-\top} (\nu^{\Gamma_0} \circ \phi^{-1})}{|\nabla \phi^{-\top} (\nu^{\Gamma_0} \circ \phi^{-1})|}.
 \]
We next observe with the decomposition 
$
 \nabla \phi = \nabla_{\Gamma_0} \phi + \frac{\partial \phi}{\partial \nu_{\Gamma_0}} \otimes \nu^{\Gamma_0}
$
and the orthogonality result $\nabla_{\Gamma_0}\phi ^{\top} (\nu^{\Gamma}\circ \phi) =0$ that
\[
\nabla \phi^{\top} (\nu^{\Gamma}\circ \phi) = \left( \frac{\partial \phi}{\partial \nu_{\Gamma_0}} \cdot \left( \nu^{\Gamma}\circ \phi \right) \right) \nu^{\Gamma_0}.
\]
Since $\phi$ maps the boundary $\Gamma_0$ onto $\Gamma$, it follows that
$\frac{\partial \phi}{\partial \nu_{\Gamma_0}} \cdot ( \nu^{\Gamma}\circ \phi) > 0 $
and thus 
\[
 \frac{\partial u}{\partial \nu_{\Gamma}}\circ \phi = 
 \left( \frac{\partial \phi}{\partial \nu_{\Gamma_0}} \cdot \nu^{\Gamma}\circ \phi \right) \nabla \hat{u} \cdot G^{-1} \nu^{\Gamma_0}.
\]
We now continue by showing that the normal component of $\frac{\partial \phi}{\partial \nu_{\Gamma_0}}$ can be expressed as the ratio between 
the bulk $\sqrt{g}$ and the surface area element $\sqrt{g_{\Gamma_0}}$. This will be achieved in the context of exterior algebras.

Let $\tau_1,...,\tau_n$ represent an orthonormal basis of the tangent space $T\Gamma_0$ and thus $\{\tau_1,...,\tau_n, \nu^{\Gamma_0}\}$ forms a basis of $\re^{n+1}$.
The determinant of linear map corresponding to $\nabla\phi$ evaluated on the boundary $\Gamma_0$ can be expressed in the notation of exterior algebras as follows
 \begin{align*}
  \text{det}(\nabla \phi) \tau_1 \wedge... \wedge \tau_n \wedge \nu^{\Gamma_0} 
  &= \nabla \phi \tau_1 \wedge ... \wedge \nabla \phi \tau_n \wedge \nabla \phi \nu^{\Gamma_0}
  = \nabla_{\Gamma_0} \phi \tau_1 \wedge ... \wedge \nabla_{\Gamma_0} \phi \tau_n \wedge \frac{\partial \phi}{\partial \nu_{\Gamma_0}}.
 \intertext{Since  $\nabla_{\Gamma_0}\phi\tau_1,..., \nabla_{\Gamma_0}\phi \tau_n$ form a basis of the tangent space $T\Gamma$ and
  the exterior product of any set of linearly dependent vectors is zero, we are therefore able to remove the tangent 
 component of the normal derivative yielding} 
  &= \left(\frac{\partial \phi}{\partial \nu_{\Gamma_0}} \cdot (\nu^{\Gamma}\circ \phi)\right) \nabla_{\Gamma_0} \phi \tau_1 \wedge ... \wedge \nabla_{\Gamma_0} \phi \tau_n 
  \wedge  \nu^{\Gamma}\circ\phi.
  \intertext{Observing that each term in the above exterior product is the image of the basis $\{\tau_1,...,\tau_n, \nu^{\Gamma_0}\}$ under the linear 
  mapping $\nabla_{\Gamma_0}\phi + \left(\nu^{\Gamma}\circ \phi\right) \otimes \nu^{\Gamma_0}$ gives}
  &=  \left(\frac{\partial \phi}{\partial \nu_{\Gamma_0}} \cdot \left(\nu^{\Gamma}\circ \phi\right)\right) 
  \text{det}\left( \nabla_{\Gamma_0} \phi + \left(\nu^{\Gamma}\circ \phi \right) \otimes \nu^{\Gamma_0} \right) \tau_1 \wedge ... \wedge \tau_n \wedge  \nu^{\Gamma_0}.
 \end{align*}
 Hence it follows 
 \[
  \text{det} \nabla \phi = \left( \frac{\partial \phi}{\partial \nu_{\Gamma_0}} \cdot \left( \nu^{\Gamma_0}\circ \phi\right) \right) \text{det}
  \left( \nabla_{\Gamma_0}\phi + \left( \nu^{\Gamma}\circ \phi \right) \otimes \nu^{\Gamma_0}\right).
 \]
We thus obtain the stated result with the following observations
\begin{align*}
 \left(\text{det} \nabla \phi\right)^2 &= \text{det} \nabla \phi^{\top} \nabla \phi = g\\
 \left( \text{det} \left(\nabla_{\Gamma_0} \phi +  \left(\nu^{\Gamma}\circ \phi \right) \otimes \nu^{\Gamma_0}\right)\right)^2
&= \text{det} \left(\left(\nabla_{\Gamma_0}\phi + \left( \nu^{\Gamma}\circ \phi\right)\otimes \nu^{\Gamma_0}\right)^{\top} 
\left(\nabla_{\Gamma_0}\phi + \left( \nu^{\Gamma}\circ \phi\right)\otimes \nu^{\Gamma_0}\right) \right)\\
&= \text{det}\left( \nabla_{\Gamma_0} \phi^{\top}\nabla_{\Gamma_0}\phi + \nu^{\Gamma_0}\otimes \nu^{\Gamma_0} \right)= g_{\Gamma_0}.
\end{align*}

\end{proof}


\section{First applications of the domain mapping method to complex random geometries involving random surfaces}
We will now consider two model elliptic problems posed on complex random domains involving random surfaces. 
In particular, the first problem will be posed on a sufficiently smooth random surface and the second on a random bulk-surface.
In both cases, the complete random domain mapping will be assumed to be known. Furthermore, we will assume 
that the computational domain was chosen to coincide with the expected domain, and thus will assume in both cases that $\mathbb{E}[\phi] = 0$. 
We will now employ the domain mapping method,
and reformulate both equations onto their corresponding expected domain 
and prove well-posedness as well as a regularity result.

\subsection{An elliptic equation on a random surface}
Let $\Gamma(\omega) $ represent a random, compact $C^2-$hypersurface in $\re^{n+1}$ prescribed by
\begin{equation}
 \Gamma(\omega) = \{ \phi(\omega,x) \, | \, x \in \Gamma_0  \}
\end{equation}
for a given random field $\phi \in L^{\infty}(\Omega; C^2(\Gamma_0; \re^{n+1}))$ defined over a fixed, compact $C^2-$hypersurface $\Gamma_0\subset \re^{n+1}$.
We will assume that the random domain mapping
$
 \phi(\omega, \cdot): \Gamma_0 \rightarrow \Gamma (\omega)
$
is a $C^2-$ diffeomorphism for almost every $\omega$ and furthermore satisfies the uniform bounds
\begin{equation}\label{assump:1}
\| \phi(\omega, \cdot)\|_{C^2(\Gamma_0)}, \|\phi^{-1}(\omega,\cdot) \|_{C^2(\Gamma(\omega))} < C 
\end{equation}
for some constant $C>0$ independent of $\omega$. 
 We consider the following model elliptic equation on the random surface 
\begin{equation}\label{eq:PDEGraph1}
 -\Delta_{\Gamma(\omega)}u(\omega) + u(\omega) = f(\omega) \quad \text{on } \Gamma(\omega)
\end{equation}
for a given random field
$
f(\omega,\cdot): \Gamma(\omega) \rightarrow \re.
$ 
Our goal is to analyse the mean solution defined by
$$
 QoI[u] := \mathbb{E}[u\circ \phi] \quad \text {on } \Gamma_0.
$$
Reformulating (\ref{eq:PDEGraph1}) onto the expected domain with the calculation of the Laplace-Beltrami operator provided in Lemma \ref{theorem:surfGrad} yields
\begin{equation}\label{eq:reformSurf}
 -\frac{1}{\sqrt{g_{\Gamma_0}(\omega)}} \nabla_{\Gamma_0} \cdot \left( \sqrt{g_{\Gamma_0}(\omega)} G_{\Gamma_0}^{-1}(\omega) \nabla_{\Gamma_0}\hat{u}(\omega)\right)
 + \hat{u}(\omega) = \hat{f}(\omega) \quad \text{on } \Gamma_0,
\end{equation}
where the random coefficient is given by
\begin{align}\label{eq:CoefGSurf1}
     G_{\Gamma_0}(\omega) &= \nabla_{\Gamma_0}\phi^{\top}(\omega)\nabla_{\Gamma_0} \phi(\omega) + \nu^{\Gamma_0}\otimes \nu^{\Gamma_0}
\end{align}
with $g_{\Gamma_0}(\omega)= \det G_{\Gamma_0}(\omega)$. Multiplying through by surface area element $\sqrt{g_{\Gamma_0}(\omega)}$ and integrating by parts, we 
arrive at the following mean-weak formulation on the fixed deterministic domain $\Gamma_0$. 
\begin{problem}[Mean-weak formulation]\label{p3}
Given $\hat{f} \in L^2(\Omega; L^2(\Gamma_0))$, find
$\hat{u} \in  L^2(\Omega;H^1(\Gamma_0))$  such that
\begin{equation}\label{prob:MeanWeakSurf}
\int_\Omega \int_{\Gamma_0} \mathcal{D}_{\Gamma_0}(\omega)\nabla_{\Gamma_0} \hat{u}(\omega) \cdot \nabla_{\Gamma_0} \hat{\varphi}(\omega)
+ \hat{u}(\omega) \hat{\varphi}(\omega)\sqrt{g_{\Gamma_0}(\omega)} = \int_\Omega \int_{\Gamma_0} \hat{f}(\omega) \hat{\varphi}(\omega)\sqrt{g_{\Gamma_0}(\omega)}
\end{equation}
for every $\hat{\varphi} \in L^2(\Omega; H^1(\Gamma_0))$. Here, we have set 
$
\mathcal{D}_{\Gamma_0}(\omega) = \sqrt{g_{\Gamma_0}(\omega)} G_{\Gamma_0}^{-1}(\omega). 
$
\end{problem}
We denote the associated bilinear form $a(\cdot,\cdot): L^2(\Omega;H^1(\Gamma_0))\times L^2(\Omega;H^1(\Gamma_0)) \rightarrow \re $ and linear functional 
$l(\cdot): L^2(\Omega; L^2(\Gamma_0))\rightarrow \re$ by 
\begin{align}
 a(\hat{u},\hat{\varphi}) &= \int_\Omega \int_{\Gamma_0} \mathcal{D}_{\Gamma_0}(\omega)\nabla_{\Gamma_0} \hat{u}(\omega) \cdot \nabla_{\Gamma_0} \hat{\varphi}(\omega) 
 + \hat{u}(\omega) \hat{\varphi}(\omega)\sqrt{g_{\Gamma_0}(\omega)}\\
 l(\hat{\varphi}) &= \int_\Omega \int_{\Gamma_0} \hat{f}(\omega) \hat{\varphi}(\omega)\sqrt{g_{\Gamma_0}(\omega)}.
\end{align}
Thus the mean-weak formulation can be written more succiently as 
\begin{equation}
 a(\hat{u},\hat{\varphi}) = l(\hat{\varphi}) \quad \text{for all } \hat{\varphi} \in L^2(\Omega; H^1(\Gamma_0)).  
\end{equation}

\begin{proposition}\label{Prop:BoundsSurfCoef}
 Under the uniformity assumptions (\ref{assump:1}) on the random domain mapping, there exists constants $C_{D_{\Gamma_0}}, C_{g_{\Gamma_0}} >0$ such that the singular 
 values $\sigma_i$ of $\mathcal{D}_{\Gamma_0}$ and the surface area element $\sqrt{g_{\Gamma_0}}$ are bounded above and below by 
 \begin{align}\label{eq:uniformBoundSurfCoef1}
  0 &< C_{D_{\Gamma_0}}^{-1} \leq \sigma_i \left( \mathcal{D}_{\Gamma_0}(\omega,x) \right) \leq C_{D_{\Gamma_0}} < + \infty\\
  \label{eq:uniformBoundSurfCoef2}
   0 &< C_{g_{\Gamma_0}}^{-1} \leq \sqrt{g_{\Gamma_0}(\omega,x)} \leq C_{g_{\Gamma_0}} < + \infty  
 \end{align}
  for all  $x \in \Gamma_0 \text{ and a.e. } \omega$.
\end{proposition}

\begin{proof}
We can rewrite $G_{\Gamma_0}$ using the orthogonality $\nabla_{\Gamma_0}\phi^{\top}(\nu^{\Gamma}\circ \phi) =0,$
as follows
 \begin{align*}
  G_{\Gamma_0} = \nabla_{\Gamma_0}\phi^{\top} \nabla_{\Gamma_0} \phi + \nu^{\Gamma_0}\otimes \nu^{\Gamma_0}
  = \left(\nabla_{\Gamma_0} \phi + \nu^{\Gamma}\circ \phi \otimes \nu^{\Gamma_0} \right)^{\top}
   \left(\nabla_{\Gamma_0} \phi + \nu^{\Gamma}\circ \phi \otimes \nu^{\Gamma_0} \right).
 \end{align*}
Examining each term separately, we see that the inverse is given by
\[
 \left(\nabla_{\Gamma_0} \phi + \nu^{\Gamma}\circ \phi \otimes \nu^{\Gamma_0} \right)^{-1} 
 = \nabla_{\Gamma}\phi^{-1}\circ \phi + \nu^{\Gamma_0}\otimes \nu^{\Gamma} \circ \phi.
\]
Hence it follows 
\begin{align*}
 G_{\Gamma_0}^{-1} &=  \left(\nabla_{\Gamma}\phi^{-1}\circ \phi + \nu^{\Gamma_0}\otimes \nu^{\Gamma} \circ \phi \right)
		       \left(\nabla_{\Gamma}\phi^{-1} \circ \phi + \nu^{\Gamma_0}\otimes \nu^{\Gamma} \circ \phi \right)\\
		       &= \left( \nabla_{\Gamma} \phi^{-1} \circ \phi \right) \left( \nabla_{\Gamma} \phi^{-\top} \circ \phi \right) + \nu^{\Gamma_0}\otimes \nu^{\Gamma_0}.
\end{align*}
Therefore with (\ref{assump:1}), we have uniform bounds above and below on the singular values of $G_{\Gamma_0}(\omega)$ and hence
obtain the estimates (\ref{eq:uniformBoundSurfCoef1}) and (\ref{eq:uniformBoundSurfCoef2}). 
\end{proof}

A direct consequence of the above uniform bounds on the random coefficients is the existence and uniqueness of a solution to (\ref{prob:MeanWeakSurf})
gauranteed by the Lax-Milgram theorem.
\begin{theorem}
 Given any $\hat{f}\in L^2(\Omega; L^2(\Gamma_0))$, there exists a unique solution $\hat{u}$ to the mean-weak formulation (\ref{prob:MeanWeakSurf}) that satisfies the energy estimate
 \begin{equation}
 \label{eq:SurfStabEst}
 \|\hat{u}\|_{L^2(\Omega; H^1(\Gamma_0))} \leq c \|\hat{f}\|_{L^2(\Omega; L^2(\Gamma_0))}.
 \end{equation}
\end{theorem}

\begin{proof}
The stability estimate (\ref{eq:SurfStabEst}) follows from the coercivity of $a(\cdot,\cdot)$. 
\end{proof}

By considering the original surface equation (\ref{eq:PDEGraph1}) on $\Gamma(\omega)\in C^2$, we would expect from standard 
elliptic surface regularity results that for given $f(\omega) \in L^2(\Gamma(\omega))$, the pathwise solution belongs to $u(\omega) \in H^2(\Gamma(\omega))$
and therefore $\hat{u}(\omega) \in H^2(\Gamma_0)$ for a.e. $\omega$. 
However since the $H^2$ a-priori estimate on $u(\omega)$ will naturally depend on the geometry of the realisation $\Gamma(\omega)$, it is not immediately clear whether 
the solution to
the mean-weak formulation belongs to $\hat{u} \in L^2(\Omega;H^2(\Gamma_0))$. 
We will therefore continue by explicitly treating all arising constants and their dependency on the geometry of the random domain.

\begin{theorem}[Regularity]\label{theorem:RegularitySurfMean}
 Given any $\hat{f} \in L^2(\Omega; L^2(\Gamma_0))$, the solution to (\ref{prob:MeanWeakSurf}) belongs to $\hat{u} \in L^2(\Omega; H^2(\Gamma_0))$
 and furthermore satisfies the following estimate
 \begin{equation}\label{eq:aprioriSurf}
 \| \hat{u} \|_{L^2(\Omega; H^2(\Gamma_0))} \leq C \| \hat{f} \|_{L^2(\Omega; L^2(\Gamma_0))}.
 \end{equation}
 
 
 \begin{proof}
  Let us consider the push-forward $u = \hat{u} \circ \phi^{-1}$ of realisations of the weak solution onto $\Gamma(\omega)$ for almost every $\omega$,
  which as a result of the tensor structure $L^2(\Omega; H^1(\Gamma_0)) \cong L^2(\Omega) \otimes H^1(\Gamma_0)$ is a pathwise weak solution of 
  \begin{equation}\label{eq:PDEpushForward}
   - \Delta_{\Gamma(\omega)} u(\omega) + u(\omega) = f(\omega) \quad \text{on } \Gamma(\omega)
  \end{equation}
with $f = \hat{f} \circ \phi^{-1}.$ Since for almost every $\omega \in \Omega$, 
$\Gamma(\omega)$ is $C^2$ and $f(\omega) \in L^2(\Gamma(\omega))$, it follows that 
$u(\omega) \in H^2(\Gamma(\omega))$ and therefore $\hat{u}(\omega) \in H^2(\Gamma_0)$. 
For the a-priori estimate (\ref{eq:aprioriSurf}), it was shown in \cite{dziuk2013finite} through a series of integration by parts and interchanging of tangential derivatives
that the $H^2$ semi-norm satisfies
\[
 | u(\omega) |_{H^2(\Gamma(\omega))} \leq \| \Delta_{\Gamma(\omega)} u(\omega) \|_{L^2(\Gamma(\omega))} + c(\omega) |u(\omega)|_{H^1(\Gamma(\omega))}
\]
with 
$$
c(\omega) = \sqrt{ \| H^{\Gamma(\omega)} \mathcal{H}^{\Gamma(\omega)}  -2 \left( \mathcal{H}^{\Gamma(\omega)} \right)^2  \|_{L^{\infty}(\Gamma(\omega))}      }.
$$
Here $H^{\Gamma(\omega)} = trace\left( \mathcal{H}^{\Gamma(\omega)} \right)$ is the mean-curvature. 
Hence with the uniform bounds (\ref{assump:1}) on the random domain mapping and the previously calculated expression (\ref{eq:Weingarten}) for the Weingarten map, we obtain 
an upper bound on the constant $c(\omega)$ independent of $\omega$. 
Thus, with the PDE (\ref{eq:PDEpushForward}) pointwise  we have the bound
\[
 \|u(\omega) \|_{H^2(\Gamma(\omega))} \leq c \left(\|f(\omega) \|_{L^2(\Gamma(\omega))} +  \|u(\omega) \|_{H^1(\Gamma(\omega))}\right).
\]
We can now pull-back onto the expected domain, applying the norm equivalence of the pull-back transformation
\[
 C^{-1}\|\hat{u}(\omega) \|_{H^k(\Gamma_0)} \leq \|u(\omega)\|_{H^k(\Gamma(\omega))} \leq C\| \hat{u}(\omega) \|_{H^k(\Gamma_0)} \quad \text{for } k=0,1,2 \text{ and a.e. } \omega
\]
where the constants are independent of $\omega$ due to bounds (\ref{assump:1}), and the stability estimate (\ref{eq:SurfStabEst}) to obtain
\[
 \|\hat{u}(\omega)\|_{H^2(\Gamma_0)} \leq C \|\hat{f} (\omega) \|_{L^2(\Gamma_0)}.
\]
and thus the stated result.
 \end{proof}

\end{theorem}

\subsection{A coupled elliptic system on a random bulk-surface}
For the second problem, we consider a coupled elliptic system on a random bulk-surface motivated by the deterministic case 
analysed in \cite{elliott2012finite}.
More precisely, the geometric setting is as follows.
We let $\{\Gamma(\omega)\}$ denote a family of random, compact $C^2-$hypersurfaces in $\re^{n+1}$ enclosing open domains 
$D(\omega)$ and will denote the outer unit normal by $\nu^{\Gamma(\omega)}.$
The family of random domains will be prescribed by the mapping 
\begin{equation}
 \phi: \overline{D_0} \rightarrow \overline{D(\omega)} \quad \phi|_{\Gamma_0}: \Gamma_0 \rightarrow \Gamma(\omega),
\end{equation}
where the reference surface $\Gamma_0\subset \re^{n+1}$ will also be a compact $C^2-$hypersurface with open interior $D_0$. 
We will assume that the domain mapping is a $C^2-$diffeomorphism for $a.e.$  $\omega \in \Omega$ and additionally satisfies
\begin{equation}\label{assumptionMapping2}
 \|\phi(\omega, \cdot) \|_{C^2(\overline{D_0})}, \quad \|\phi^{-1}(\omega, \cdot)\|_{C^2(\overline{D(\omega)})} < C
\end{equation}
for a constant $C>0$ independent of $\omega$. 
The proposed coupled elliptic system on the random bulk-surface reads as follows
\begin{subequations}\label{eq:coupledEllipticRandom}
\begin{align}
 -\Delta u(\omega) + u(\omega) &= f(\omega) \quad \text{on } D(\omega)\\
 \alpha u(\omega) -\beta v(\omega) + \frac{\partial u}{\partial \nu_{\Gamma}}(\omega) &= 0 \quad \text{on } \Gamma(\omega)\\
 -\Delta_{\Gamma(\omega)}v(\omega) +v(\omega)  + \frac{\partial u}{\partial \nu_{\Gamma}}(\omega) &= f_{\Gamma}(\omega) \quad \text{on } \Gamma(\omega).
\end{align}
\end{subequations}
Here $\alpha, \beta>0$ are given positive constants and $
 f(\omega,\cdot):D(\omega) \rightarrow \re$ and $f_{\Gamma}(\omega, \cdot):\Gamma(\omega) \rightarrow \re$ are prescribed random fields.
As with our previous problem, our quantity of interest is the mean solution, that is the pair 
$\left( \mathbb{E}[u], \mathbb{E}[v] \right)$ defined by
$$
\mathbb{E}[u]:= \mathbb{E}[u\circ\phi]  \quad \mathbb{E}[v] := \mathbb{E}[v \circ \phi].
$$
Let us continue by reformulating the system (\ref{eq:coupledEllipticRandom}) onto the expected domain $\overline{D_0}$
with our previously calculated expressions for the Laplace-Beltrami operator (\ref{eq:Laplacebeltrami}) and the normal derivative (\ref{eq:NormalDeriv})
giving
\begin{subequations}\label{eq:coupledElliptic}
\begin{align}
\label{eq:coupledElliptic_bulk}
 -\frac{1}{\sqrt{g(\omega)}} \nabla \cdot \left( \sqrt{g(\omega)} G^{-1}(\omega) \nabla \hat{u}(\omega) \right) + \hat{u}(\omega) &= \hat{f}(\omega) \quad \text{in } D_0\\
 \label{eq:coupledElliptic_bc}
 \alpha \hat{u}(\omega) - \beta \hat{v}(\omega) 
 + \frac{\sqrt{g(\omega)}}{\sqrt{g_{\Gamma_0}(\omega)}} G^{-1}(\omega)\nu^{\Gamma_0} \cdot \nabla \hat{u}(\omega) &= 0 \quad \text{on } \Gamma_0\\
 \label{eq:coupledElliptic_surf}
 - \frac{1}{\sqrt{g_{\Gamma_0}(\omega)}} \nabla_{\Gamma_0} \cdot \left( \sqrt{g_{\Gamma_0}(\omega)} G_{\Gamma_0}^{-1}(\omega) \nabla_{\Gamma_0}\hat{v}(\omega)\right) + \hat{v}(\omega)  
 +\frac{\sqrt{g(\omega)}}{\sqrt{g_{\Gamma_0}(\omega)}}& G^{-1}(\omega)\nu^{\Gamma_0} \cdot \nabla \hat{u} 
 = \hat{f}_{\Gamma_0}(\omega) \quad \text{on } \Gamma_0.
\end{align}
\end{subequations}
Here the random coefficients are 
$$ G(\omega) = \nabla \phi^{\top}(\omega) \nabla \phi(\omega) \quad
 G_{\Gamma_0}(\omega) = \nabla_{\Gamma_0}\phi^{\top}(\omega) \nabla_{\Gamma_0}\phi(\omega) + \nu^{\Gamma_0} \otimes \nu^{\Gamma_0}
 $$
%
with $g(\omega) = detG(\omega)$, $g_{\Gamma_0}(\omega) = detG_{\Gamma_0}(\omega)$. For convenience, we have 
set $\hat{f}_{\Gamma_0} = f_{\Gamma} \circ \phi$. 
To derive a mean-weak formulation, we follow the 
variational approach presented in \cite{elliott2012finite}. 
We begin by multiplying through the bulk equation (\ref{eq:coupledElliptic_bulk}) by the area element $\sqrt{g}$ and integrating by parts which gives
\begin{align}\label{eq:meanDeriv1}
 \int_{D_0} \sqrt{g(\omega)} G^{-1}(\omega) \nabla \hat{u}(\omega) &\cdot \nabla \hat{\varphi}(\omega) + \hat{u}(\omega) \hat{\varphi}(\omega) \sqrt{g(\omega)}\\ 
 &- \int_{\Gamma_0} \left(\sqrt{g(\omega)}  G^{-1}(\omega) \nabla \hat{u}(\omega) \cdot \nu^{\Gamma_0}\right) \hat{\varphi}(\omega)
 = \int_{D_0}  \hat{f}(\omega) \hat{\varphi}(\omega)\sqrt{g(\omega)}\nonumber
 \end{align}
Similarly, for the surface equation (\ref{eq:coupledElliptic_surf}) we integrate by parts recalling that the hypersurface $\Gamma_0$ is without boundary, to obtain
 \begin{align}\label{eq:meanDeriv2}
 \int_{\Gamma_0} \sqrt{g_{\Gamma_0}(\omega)} G_{\Gamma_0}^{-1}(\omega) \nabla_{\Gamma_0}\hat{v}(\omega) \cdot \nabla_{\Gamma_0}\hat{\xi} 
 +  \hat{v}(\omega) \hat{\xi}(\omega) \sqrt{g_{\Gamma_0}(\omega)}\\
 + \int_{\Gamma_0}  \sqrt{g(\omega)} \left( G^{-1}(\omega)\nu^{\Gamma_0} \cdot \nabla \hat{u}(\omega) \right) \hat{\xi}(\omega)
 &= \int_{\Gamma_0}  \hat{f}_{\Gamma_0}(\omega) \hat{\xi}(\omega)\sqrt{g_{\Gamma_0}(\omega)}. \nonumber
\end{align}
Taking the weighted sum and substituting in the reformulated Robin boundary condition (\ref{eq:coupledElliptic_bc}), we arrive at the 
following mean-weak formulation:
\begin{problem}[Mean-weak formulation]
Given any $\hat{f} \in L^2(\Omega; L^2(D_0))$ and $\hat{f}_{\Gamma_0} \in L^2(\Omega;L^2(\Gamma_0))$, find  $\hat{u} \in L^2(\omega; H^1(D_0))$ and 
$\hat{v} \in L^2(\Omega;H^1(\Gamma_0))$ such that 
\begin{align*}\label{problem:meanCoupled}
 \alpha \int_{\Omega} \int_{D_0} \mathcal{D}(\omega) \nabla \hat{u}(\omega) &\cdot \nabla \hat{\varphi}(\omega)
+  \hat{u}(\omega) \hat{\varphi}(\omega)\sqrt{g(\omega)}\\
  + \beta \int_{\Omega} \int_{\Gamma_0} \mathcal{D}_{\Gamma_0}(\omega) \nabla_{\Gamma_0}\hat{v}(\omega) &\cdot \nabla_{\Gamma_0}\hat{\xi}(\omega)
 +  \hat{v}(\omega) \hat{\xi}(\omega)\sqrt{g_{\Gamma_0}(\omega)}\\
  + \int_{\Omega}\int_{\Gamma_0}\left( \alpha \hat{u}(\omega) - \beta \hat{v}(\omega) \right)& 
 ( \alpha \hat{\varphi}(\omega) - \beta \hat{\xi}(\omega) )  \sqrt{g_{\Gamma_0}(\omega)} \\
&= \alpha \int_{\Omega} \int_{D_0} \hat{f}(\omega) \hat{\varphi}(\omega) \sqrt{g(\omega)} 
 + \beta \int_{\Omega} \int_{\Gamma_0}  \hat{f}_{\Gamma_0}(\omega) \hat{\xi}(\omega)\sqrt{g_{\Gamma_0}(\omega)}.
\end{align*}
for every $\hat{\varphi} \in L^2(\Omega; H^1(D_0))$ and  $\hat{\xi} \in L^2(\Omega;H^1(\Gamma_0)).$ 
Here we set 
$\mathcal{D}_{\Gamma_0}(\omega) =\sqrt{g_{\Gamma_0}(\omega)}G_{\Gamma_0}^{-1}(\omega)$.
\end{problem}
We denote the associated bilinear form and linear functional stated above by
\begin{equation}
 a(\cdot,\cdot) : L^2(\Omega;V) \times L^2(\Omega;V) \rightarrow \re, \quad l(\cdot): L^2(\Omega;H) \rightarrow \re,
\end{equation}
where we have set $H= L^2(D_0) \times L^2(\Gamma_0)$ and $ V= H^1(D_0) \times H^1(\Gamma_0)$ to be Hilbert spaces equipped with respective inner products
\begin{align*}
 (( \hat{u}, \hat{v}), (\hat{\varphi}, \hat{\xi}))_{H} &= (\hat{u}, \hat{\varphi})_{L^2(D_0)}
+ (\hat{v}, \hat{\xi})_{L^2(\Gamma_0)} , \\
 (( \hat{u}, \hat{v}), (\hat{\varphi}, \hat{\xi}))_{V} &= (\hat{u}, \hat{\varphi})_{H^1(D_0)}
+ (\hat{v}, \hat{\xi})_{H^1(\Gamma_0)}. 
 \end{align*}
The mean-weak formulation thus reads as follows
 \begin{equation}\label{problem:meanCoupled}
 a( (\hat{u}, \hat{v}), ( \hat{\varphi}, \hat{\xi})) = l( (\hat{\varphi}, \hat{\xi})). 
\end{equation}
The following uniform bounds on the random bulk coefficients follow immediately from the assumption (\ref{assumptionMapping2}) on the random domain mapping. 
Furthermore, the derived bounds
on the surface coefficients presented in Proposition \ref{Prop:BoundsSurfCoef} also hold since the tangential derivatives of the surface parametrisation 
and its inverse are also uniformly bounded as a consequence of (\ref{assumptionMapping2}).

\begin{proposition}[Uniform bounds]
There exists constants $C_g, C_D>0$ such that the bulk area element $\sqrt{g(\omega)}$ and the singular values $\sigma_i$ of $D(\omega)$ are
uniformly bounded for all $x\in D_0$ and a.e. $\omega$ by
\begin{align}\label{eq:UniformBoundBulkCoef}
  0 < C_g^{-1} &\leq \sqrt{g(\omega, x)}\leq C_g < + \infty\\
   0 < C_D^{-1} &\leq \sigma_i \left( D(\omega, x) \right) \leq C_D < + \infty.
\end{align}
\end{proposition}

\begin{theorem}
Given any $(\hat{f}, \hat{f}_{\Gamma_0}) \in H$, there exist a unique solution $(\hat{u}, \hat{v}) \in L^2(\Omega;V)$ to 
(\ref{problem:meanCoupled}) which satisfies the energy estimate
\begin{equation}\label{eq:CoupledStabEst}
\|(\hat{u}, \hat{v})\|_{L^2(\Omega; V)} \leq c \|(\hat{f}, \hat{f}_{\Gamma_0})\|_{L^2(\Omega; H)}.
\end{equation}
\end{theorem}

\begin{proof}
With our uniform bounds (\ref{eq:UniformBoundBulkCoef}), (\ref{eq:uniformBoundSurfCoef1}) 
on the random bulk and surface coefficients, we can now proceed in verifying all the conditions of the Lax-Milgram 
theorem are satisified. 
 For a coercivity estimate, we argue
\begin{align*}
 a( (\hat{u}, \hat{v}), (\hat{u},\hat{v}))
 &\geq 
 \alpha_\text{min}\left(C_D^{-1}, C_g^{-1} \right) \|\hat{u}\|_{L^2(\Omega;H^1(D_0))}^2  
 + \beta_\text{min}\left( C_{D_{\Gamma_0}}^{-1}, C_{g_{\Gamma_0}}^{-1} \right)\| \hat{v} \|_{L^2(\Omega;H^1(\Gamma_0))}^2\\
 &+ C_{g_{\Gamma_0}}^{-1} \|\alpha \hat{u} - \beta \hat{v}\|_{L^2(\Omega; L^2(\Gamma_0))}^2\\
 &\geq C ( \|\hat{u}\|_{L^2(\Omega; H^1(D_0))}^2 + \| \hat{v}\|_{L^2(\Omega; H^1(\Gamma_0))}^2) \\
 &=C\|(\hat{u}, \hat{v})\|_{L^2(\Omega; V)}^2.
\end{align*}
For the continuity of the bilinear form $a(\cdot,\cdot)$, we apply the Cauchy-Schwarz inequality with the boundedness of the trace operator 
$\|f\|_{L^2(\Gamma_0)}\leq c_T \| f \|_{H^1(D_0)}$ as follows
\begin{align*}
 |a((\hat{u}, \hat{v})&, (\hat{\varphi},\hat{\xi}))| \\
 \leq& 
 \alpha_\text{max}(C_D, C_g) \|\hat{u}\|_{L^2(\Omega;H^1(D_0))} \|\hat{\varphi}\|_{L^2(\Omega;H^1(D_0))}
 +\beta_\text{max}(C_{D_{\Gamma_0}}, C_{g_{\Gamma_0}})\|\hat{v}\|_{L^2(\Omega;H^1(\Gamma_0))} \| \hat{\xi}\|_{L^2(\Omega;H^1(\Gamma_0))}\\
  &+ C_{g_{\Gamma_0}}\| \alpha \hat{u} - \beta \hat{v} \|_{L^2(\Omega; L^2(\Gamma_0))} \|\alpha \hat{\varphi} - \beta \hat{\xi}\|_{L^2(\Omega;L^2(\Gamma_0))}\\ 
 \leq&
 C \|(\hat{u}, \hat{v})\|_{L^2(\Omega; V)}\|(\hat{\varphi},\hat{\xi})\|_{L^2(\Omega;V)}\\
 &+ C_{g_{\Gamma_0}}  
 \left( \alpha c_T \|\hat{u}\|_{L^2(\Omega; H^1(D_0))} + \beta \|\hat{v}\|_{L^2(\Omega; L^2(\Gamma_0))} \right)
 \left( c_T \|\hat{\varphi}\|_{L^2(\Omega; H^1(D_0))} + \|\hat{\xi}\|_{L^2(\Omega; L^2(\Gamma_0))} \right)\\
 \leq&
  C \|(\hat{u}, \hat{v})\|_{L^2(\Omega;V)} \|(\hat{\varphi},\hat{\xi})\|_{L^2(\Omega; V)}.
\end{align*}
Thus we have the existence and uniqueness of a solution to (\ref{problem:meanCoupled}). The estimate (\ref{eq:CoupledStabEst}) then follows from
coercivity of $a(\cdot,\cdot).$
\end{proof}

\begin{theorem}[Regularity]\label{theorem:RegCoupled}
Given any $\hat{f} \in L^2(\Omega; L^2(D_0))$ and $\hat{f}_{\Gamma_0}\in L^2(\Omega; L^2(\Gamma_0))$, the mean-weak solution 
$(\hat{u}, \hat{v} ) $ to (\ref{problem:meanCoupled}) satisfies 
\begin{equation}
 \hat{u} \in L^2(\Omega; H^2(D_0)) \quad \hat{v} \in L^2(\Omega; H^2(\Gamma_0)). 
\end{equation}
Furthermore, we have
\begin{equation}
 \|(\hat{u}, \hat{v} ) \|_{L^2(\Omega; H^2(D_0)\times H^2(\Gamma_0))} \leq C \|( \hat{f},\hat{f}_{\Gamma_0}) \|_{L^2(\Omega;L^2(D_0) \times L^2(\Gamma_0))},
\end{equation}
where the constant $C>0$ depends only the geometry of the reference domain $\overline{D_0}$ and the uniform bound (\ref{assumptionMapping2})
on the random domain mapping.
\end{theorem}

 \begin{proof}
 Observe that for $a.e.$ $\omega \in \Omega$, the solution $(\hat{u},\hat{v})$ satisfies for every $\hat{\varphi} \in H^1(D_0)$ and $\hat{\xi}\in H^1(\Gamma_0)$,
\begin{align*}
 \alpha  \int_{D_0} \mathcal{D}(\omega) \nabla &\hat{u}(\omega) \cdot \nabla \hat{\varphi}
+  \hat{u}(\omega) \hat{\varphi}\sqrt{g(\omega)}
  + \beta  \int_{\Gamma_0} \mathcal{D}_{\Gamma_0}(\omega) \nabla_{\Gamma_0}\hat{v}(\omega) \cdot \nabla_{\Gamma_0}\hat{\xi}
 +  \hat{v}(\omega) \hat{\xi}\sqrt{g_{\Gamma_0}(\omega)}\\
  &+ \int_{\Gamma_0}\left( \alpha \hat{u}(\omega) - \beta \hat{v}(\omega) \right)
 ( \alpha \hat{\varphi} - \beta \hat{\xi})  \sqrt{g_{\Gamma_0}(\omega)} 
= \alpha  \int_{D_0} \hat{f}(\omega) \hat{\varphi} \sqrt{g(\omega)} 
 + \beta  \int_{\Gamma_0}  \hat{f}_{\Gamma_0}(\omega) \hat{\xi}\sqrt{g_{\Gamma_0}(\omega)}.
\end{align*}
 Setting $\hat{\varphi} = 0$ gives
\begin{align*}
   \beta  \int_{\Gamma_0} \mathcal{D}_{\Gamma_0}(\omega) \nabla_{\Gamma_0}&\hat{v}(\omega) \cdot \nabla_{\Gamma_0}\hat{\xi}
 +  \hat{v}(\omega) \hat{\xi}\sqrt{g_{\Gamma_0}(\omega)}\\
  - \int_{\Gamma_0}&\left( \alpha \hat{u}(\omega) - \beta \hat{v}(\omega) \right)
  \beta \hat{\xi}  \sqrt{g_{\Gamma_0}(\omega)} 
=  \beta  \int_{\Gamma_0}  \hat{f}_{\Gamma_0}(\omega) \hat{\xi}\sqrt{g_{\Gamma_0}(\omega)}.
\end{align*}
Hence we see that $\hat{v}(\omega)$ is the pathwise weak solution to the elliptic surface equation
\begin{align*}
 - \beta \nabla_{\Gamma_0} \cdot \left( \mathcal{D}_{\Gamma_0}(\omega) \nabla_{\Gamma_0}\hat{v}(\omega) \right)
 + ( \beta + \beta^2) \sqrt{g_{\Gamma_0}(\omega)} \hat{v}(\omega) = \alpha \beta \sqrt{g_{\Gamma_0}(\omega)} \hat{u}(\omega) 
 + \beta \sqrt{g_{\Gamma_0}(\omega)} \hat{f}_{\Gamma_0}(\omega).
\end{align*}
It therefore follows form the surface regularity result given in Theorem \ref{theorem:RegularitySurfMean} since $\hat{u}(\omega) \in L^2(\Gamma_0)$, 
that $\hat{v}(\omega) \in H^2(\Gamma_0)$ for a.e. $\omega$ and furthermore
\begin{equation}\label{eq:regProof1}
 \|\hat{v}(\omega)\|_{H^2(\Gamma_0)} \leq C \left( \| \hat{f}_{\Gamma_0}(\omega) \|_{L^2(\Gamma_0)} + \|\hat{u}(\omega)\|_{L^2(\Gamma_0)} \right)
\end{equation}
where the constant $C>0$ is independent of $\omega$.
To obtain higher regularity of the bulk quantity, we set $\hat{\xi} = 0$ yielding 
\begin{align*}
 \alpha  \int_{D_0} \mathcal{D}(\omega) \nabla &\hat{u}(\omega) \cdot \nabla \hat{\varphi}
+  \hat{u}(\omega) \hat{\varphi}\sqrt{g(\omega)}
  + \int_{\Gamma_0}\left( \alpha \hat{u}(\omega) - \beta \hat{v}(\omega) \right)
  \alpha \hat{\varphi}   \sqrt{g_{\Gamma_0}(\omega)} 
= \alpha  \int_{D_0} \hat{f}(\omega) \hat{\varphi} \sqrt{g(\omega)}. 
\end{align*}
This is precisely the weak formulation of the following elliptic boundary value problem subject to the reformulated Robin boundary condition
\begin{align*}
-\alpha \nabla \cdot \left( \mathcal{D}(\omega) \nabla \hat{u}(\omega) \right) + \alpha \sqrt{g(\omega)} \hat{u}(\omega) 
&= \alpha \sqrt{g(\omega)} \hat{f}(\omega) \quad \text{in } D_0\\
\mathcal{D}(\omega) \nabla \hat{u}(\omega) \cdot \nu^{\Gamma_0} + \alpha \sqrt{g_{\Gamma_0}(\omega)} \hat{u}(\omega) 
&= \beta \sqrt{g_{\Gamma_0}(\omega)} \hat{v}(\omega) \quad \text{on } \Gamma_0.
\end{align*}
Since the coefficients are sufficiently regular, more precisely
\begin{align*}
\mathcal{D}_{ij}(\omega) &\in C^1\left( \overline{D}_0 \right),  \quad
\alpha \sqrt{g(\omega)} \in L^{\infty}(D_0), \quad
\alpha \sqrt{g(\omega)} \hat{f}(\omega) \in L^2(D_0),\\
0 < \alpha_0 &\leq \alpha \sqrt{g_{\Gamma_0}(\omega)} \in C^1(\Gamma_0),\quad
\beta \sqrt{g_{\Gamma_0}(\omega)} \hat{v}(\omega) \in H^1(\Gamma_0), 
\end{align*}
and the boundary is sufficiently smooth $\Gamma_0 \in C^2$, we can apply standard regularity results \cite{ladyzhenskaya1968nina} to deduce 
$\hat{u}(\omega) \in H^2(D_0)$ for a.e. $\omega$ with the estimate
\begin{equation}\label{eq:regProof2}
\|\hat{u}(\omega) \|_{H^2(D_0)} \leq C\left( \|\hat{f}(\omega)\|_{L^2(D_0)} + \| \hat{v}(\omega)\|_{H^1(\Gamma_0)}\right). 
\end{equation}
Here the constant $C>0$ is independent of $\omega$ since all the coefficients are uniformly bounded and furthermore, $\mathcal{D}(\omega)$ is uniformly 
elliptic in $\omega$.
Combining (\ref{eq:regProof1}) and (\ref{eq:regProof2}) with the stability estimate (\ref{eq:CoupledStabEst}) and boundedness of the trace operator leads to 
\begin{align*}
 \|\hat{u}(\omega)\|_{H^2(D_0)} + \|\hat{v}(\omega)\|_{H^2(\Gamma_0)}
  &\leq C \left( \|\hat{f}_{\Gamma_0}(\omega) \|_{L^2(\Gamma_0)} + C_T \|\hat{u}(\omega)\|_{H^1(\Gamma_0)}
   + \|\hat{f}(\omega)\|_{L^2(D_0)} + \|\hat{v}(\omega)\|_{H^1(\Gamma_0)} \right)\\
   &\leq C \left( \|\hat{f}_{\Gamma_0}(\omega) \|_{L^2(\Gamma_0)} + \|\hat{f}(\omega) \|_{L^2(D_0)}\right).
\end{align*}
and hence the stated result. 
 \end{proof}


\section{An abstract numerical analysis of elliptic equations on random curved domains}

We continue by considering in an abstract setting, the mean-weak formulation of general elliptic equations 
on random curved domains after being transformed onto the expected domain via the given 
stochastic domain mapping. 
Working in this abstract framework, we will present and analyse a finite element discretisation 
coupled with the Monte-Carlo method to approximate our
quantity of interest, the mean solution.
As the expected domain is assumed to be curved, the proposed finite element method will involve perturbations of the variational set up 
corresponding to the approximation of the domain.
An optimal error bound in the energy norm 
for our non-conforming approach is derived with the help of the first lemma of Strang with
suitable assumptions on the finite element space approximation and arising consistency error. Furthermore, an $L^2(\Omega;L^2)$-type estimate is proved 
by a standard duality argument.
\subsection{Abstract mean-weak formulation}
Let $V$ and $H$ denote separable Hilbert spaces for which the embedding $V\hookrightarrow H$ is dense and continuous.
We assume that we are in the setting where we have a sample dependent bilinear form $\tilde{a}(\omega;\cdot,\cdot): V \times V \rightarrow \re$ and linear
functional $\tilde{l}(\omega; \cdot): H \rightarrow \re$ corresponding to the path-wise weak formulation 
\[
 \tilde{a}(\omega; u(\omega), \varphi) = \tilde{l}(\omega; \varphi) 
\]
of the elliptic equation after being reformulated onto the expected domain.
For convenience, we will omit the pull-back notation for functions $\hat{u}$ since all the subsequent analysis will be considered on the expected domain. 
The mean-weak formulation will thus in general read as follows:
\begin{problem}[Mean-weak formulation]
Find $u\in L^2(\Omega; V) $ such that for every $\varphi \in L^2(\Omega;V)$ we have
\begin{equation}\label{absmean}
\int_{\Omega}\tilde{a}(\omega;u(\omega), \varphi(\omega)) \, d \mathbb{P}(\omega) = \int_{\Omega}\tilde{l} (\omega;  \varphi(\omega))\, d \mathbb{P}(\omega).
\end{equation}
\end{problem} 
We denote the associated bilinear form $a(\cdot,\cdot):L^2(\Omega;V)\times L^2(\Omega;V)\rightarrow \re$ and linear functional $l(\cdot):L^2(\Omega;H)\rightarrow \re$ by
$$a(u, \varphi) = \int_{\Omega}\tilde{a}(\omega;u(\omega), \varphi(\omega)), \quad l(\varphi) = \int_{\Omega}\tilde{l} (\omega, \varphi(\omega)).$$
and shall assume all the requirements of the Lax-Milgram theorem are satisfied thus ensuring the existence and uniqueness of the solution.
\subsection{Abstract formulation of the finite element discretisation}
For a given $h \in (0,h_0)$, let $\mathcal{V}_h$ be a finite dimensional space that will represent a finite element space and 
let $V_h$ and $H_h$  denote the space $\mathcal{V}_h$ endowed with respective norms $\|\cdot\|_{V_h}$ and $\|\cdot \|_{H_h}$.
We assume that $V_h$ and $H_h$ are Hilbert spaces and furthermore that
$V_h \hookrightarrow H_h$ is uniformly embedded, that is 
\[
 \| \chi_h \|_{H_h} \leq c \| \chi_h\|_{V_h} \quad \text{for all } \chi_h \in V_h,
\]
for a constant $c>0$ independent of $h.$
In practice, the spaces $V_h$ and $H_h$ will represent equivalent Hilbert spaces to the continuous solution spaces $V$ and $H$
but posed over a discrete approximation of the curved domain, with $h$ denoting the discretisation parameter. 
We introduce the sample-dependent bilinear form and linear functional
$$\tilde{a}_h(\omega; \cdot, \cdot): \mathcal{V}_h \times \mathcal{V}_h \rightarrow \re \quad 
\tilde{l}_h(\omega; \cdot): \mathcal{V}_h\rightarrow \re,$$
that are perturbations 
approximating their continuous counterparts and will assume $\tilde{a}_h(\omega:\cdot,\cdot)$ is uniformly $V_h$-elliptic and bounded and 
additionally $\tilde{l}_h(\omega; \cdot)$ 
is uniformly bounded. More precisely, there exists constants $c_1,c_2,c_3>0$ independent of $\omega$ and $h$ such that 
\begin{align}
\label{assumption:Coer1}
 \tilde{a}_h(\omega; \chi_h, \chi_h) &\geq c_1 \|\chi\|_{V_h}^2\\
 \label{assumption:Coer2}
 |\tilde{a}_h(\omega; \chi_h, W_h)| &\leq c_2 \|\chi_h\|_{V_h} \| W_h\|_{V_h}\\
\label{assumption:Coer3}
 |\tilde{l}(\omega; \chi_h)| &\leq c_3 \|\chi_h\|_{H_h}.
\end{align}
The finite element approximation of the mean-weak formulation (\ref{absmean}) for a given a finite dimensional subspace $\mathcal{V}_h\subset V_h$
will then take the following form:
\begin{problem}[Semi-discrete problem]
Find $U_h \in L^2(\Omega;\mathcal{V}_h)$ such that
\begin{equation}\label{abssemi}
a_h(U_h, \phi_h) =\int_{\Omega} \tilde{a}_h(\omega; U_h(\omega), \phi_h(\omega))\, d \mathbb{P}(\omega) 
= \int_{\Omega} \tilde{l}_h(\omega;  \phi_h(\omega)) \, d \mathbb{P}(\omega) = l_h(\phi_h)
\end{equation}
for all $\phi_h \in L^2(\Omega;\mathcal{V}_h).$
\end{problem}
By our uniform assumptions of the bilinear form $\tilde{a}(\omega; \cdot, \cdot)$ and the linear functional $\tilde{l}(\omega;\cdot)$,
we deduce the existence and uniqueness of a solution to the semi-discrete problem.
\begin{theorem}
There exists a unique solution $U_h\in L^2(\Omega; V_h)$ 
to the semi-discrete problem (\ref{abssemi}) that satisfies 
\begin{equation}\label{discretestab}
\|U_h\|_{L^2(\Omega; V_h)} \leq C 
\end{equation}
with the constant $C>0$ independent of $h\in(0,h_0).$
\end{theorem}
Observe that if we let $\{\chi_j\}_{j=1}^N$ be a basis of $\mathcal{V}_h$ and express 
$U_h, \phi_h \in L^2(\Omega;\mathcal{V}_h) \cong L^2(\Omega) \otimes \mathcal{V}_h$ in the form
$$U_h(\omega) = \sum\limits_{j=1}^N U_j(\omega) \chi_j \quad \phi_h(\omega) = \sum\limits_{j=1}^N \phi_j(\omega) \chi_j,$$
where $U(\omega) = (U_1(\omega),...,U_N(\omega))^{\top} \in L^2(\Omega)^N$ and  $\Phi(\omega) = (\phi_1(\omega),...,\phi_N(\omega))^{\top} \in L^2(\Omega)^N$,
then (\ref{abssemi}) can be rewritten as
\begin{equation}
\int_{\Omega} \Phi(\omega) \cdot S(\omega) U(\omega) = \int_{\Omega} \Phi(\omega) \cdot F(\omega).
\end{equation}
Here the random stiffness matrix $S(\omega)=(S_{ij}(\omega))_{i,j=1,...,N}$ and load vector $F(\omega)=(F_j(\omega))_{j=1,...,N}$ are given by
$   
S_{ij}(\omega) = \tilde{a}_h(\omega; \chi_j, \chi_i),$ $ F_j(\omega) = \tilde{l}_h(\omega;\chi_j).$
Since $\phi_j(\omega) \in L^2(\Omega)$ are arbitrary, we deduce that the semi-discrete problem is equivalent to finding $U\in L^2(\Omega; \re^N)$ which satisfies
\begin{equation}
S(\omega) U(\omega) = F(\omega) \quad \text{for a.e. } \omega.
\end{equation}


\subsection{Assumptions on the finite element approximation and the continuous equations}
\label{section:assumpList}
 We now state all the necessary assumptions that will be required in deriving an error estimate for the semi-discrete solution.
 In order to compare our semi-discrete solution with the continuous solution, we first need to assume the existence of a lifting map.
 \begin{assumption}[Lifting map]
  There exists a linear mapping $\Lambda_h: \mathcal{V}_h \rightarrow V$ for which there exists constants $c_1, c_2>0$ independent of $h\in(0,h_0)$ such that for all $\chi_h \in \mathcal{V}_h$
\begin{align*}
\tag{L1}
    c_1 \|\chi_h\|_{ H_h} &\leq\| \Lambda_h \chi_h\|_{H} \leq  c_2 \| \chi_h\|_{H_h}\\
    \tag{L2}
    c_1 \|\chi_h\|_{ V_h} &\leq\| \Lambda_h \chi_h\|_{V} \leq  c_2 \| \chi_h\|_{V_h}.
\end{align*} 
 \end{assumption}
We denote the lifted finite dimensional space by $ V_h^l:= \Lambda_h \mathcal{V}_h$. Next, we introduce 
the Hilbert space $Z_0\hookrightarrow V$ which shall represent a space consisting of functions of higher regularity 
for which we assume we have the following interpolation estimate. 
\begin{assumption}[Approximation of finite element space]
There exists a well-defined interpolation operator $I_h:  Z_0 \rightarrow V_h^l$ for which there exists $c>0$ such that
\begin{align}
\tag{I1}
    \| \eta - I_h \eta \|_H + h \|\eta - I_h \eta \|_V &\leq c h^2 \|\eta\|_{Z_0} \quad \text{for } \eta \in  Z_0.
\end{align}
\end{assumption}
Naturally, the lifting map and interpolation operator can be extended to random functions in a pathwise sense
\begin{align*}
 \left(\Lambda_h \phi_h\right)(\omega):&= \Lambda_h \phi_h(\omega)\quad
\left(I_h \phi_h\right) (\omega) := I_h\phi_h(\omega),
\end{align*}
and the previous estimates (L1),(L2), (I1) hold for their respective norms $\|\cdot\|_{L^2(\Omega;H)}$ and $\|\cdot\|_{L^2(\Omega;V)}$. 
We continue by imposing bounds on the consistency error arising from the pertubation of the variational form. 
For this, we will assume the existence of an inverse lifting map $\Lambda_h : L^2(\Omega;Z_0)\rightarrow L^2(\Omega; V_h)$ and 
will denote inverse lift of a function $w$ by $w^{-l}$.
\begin{assumption}[Consistency error]
Given any $W_h, \phi_h \in L^2(\Omega;\mathcal{V}_h)$ with corresponding lifts $w_h, \chi_h \in L^2(\Omega;V_h^l)$, we have
the bounds
\begin{align}
\tag{P1} 
|l(\varphi_h) - l_h(\phi_h)| &\leq c h^{2} \|\varphi_h\|_{L^2(\Omega;H)} \\
\tag{P2}
|a( w_h, \varphi_h) - a_h(W_h, \phi_h)| &\leq c h \|w_h\|_V \|\varphi_h\|_{L^2(\Omega;V)}.
\end{align}
Furthermore, for any $w, \varphi \in L^2(\Omega;Z_0)$ with inverse lifts $w^{-l}, \varphi^{-l}$ we have
\begin{align}
\tag{P3}
|a(w, \varphi) - a_h(w^{-l}, \varphi^{-l})| \leq ch^2 \| w\|_{L^2(\Omega;Z_0)} \|\varphi\|_{L^2(\Omega;Z_0)}.
\end{align}
\end{assumption}
Our final assumption will be on the regularity of an associated dual problem that will enable us to derive an
 $L^2(\Omega;H)$ error estimate using the standard Aubin-Nitsche trick. The associated dual problem reads as follows:
\begin{problem}[Dual problem]
For a given $g \in L^2(\Omega;H)$, find $w(g) \in L^2(\Omega; V)$ such that
\begin{equation}\label{dual}
a(\varphi, w(g)) = (g, \varphi)_{L^2(\Omega;H)} \quad \text{for } \varphi \in L^2(\Omega; V).
\end{equation}
\end{problem}
Here $(\cdot,\cdot)_{L^2(\Omega;H)}$ denotes the inner product on the Hilbert space $L^2(\Omega;H)$. 
\begin{assumption}[Regularity of dual problem]
The solution $w(g)$ to the dual problem belongs to space $L^2(\Omega;Z_0)$ and furthermore satisfies
\begin{equation}\LeftEqNo
\|w(g)\|_{L^2(\Omega;Z_0)} \leq c \|g\|_{L^2(\Omega; H)}
\tag{R1}
\end{equation}
for a constant $c>0$ independent of both $g$ and  $h\in(0,h_0)$. 
\end{assumption}

\subsection{Error estimates for the semi-discrete solution}
Recall that the abstract finite element space $\mathcal{V}_h$ is not necessarily contained in the Hilbert space $V$.
However, with the assumed existence of a lifting map
$$
 \Lambda_h : L^2(\Omega; \mathcal{V}_h) \rightarrow L^2(\Omega; V_h^l) \subset L^2(\Omega; V),
$$
we can lift the discrete bilinear form $a_h(\cdot, \cdot)$ and the linear functional $l_h(\cdot)$ onto 
the space $L^2(\Omega; V_h^l)$ 
by the following relations for $w_h = \Lambda_h W_h, \varphi_h= \Lambda_h \phi_h \in L^2(\Omega; V_h^l)$
\begin{align}
 a_h^l(w_h, \varphi_h) :&= a_h(W_h, \phi_h) \quad
 l_h^l(w_h) := l_h(W_h),
\end{align}
thus inducing a third variational problem equivalent to (\ref{abssemi}).
\begin{problem}[Lifted semi-discrete problem]
Find $u_h \in L^2(\Omega; V_h^l)$ such that for every $\varphi_h \in L^2(\Omega; V_h^l)$ we have
\begin{equation}\label{eq:LiftedSemi}
 a_h^l(u_h, \varphi_h) = l_h^l(\varphi_h).
 \end{equation}
\end{problem}
Since $L^2(\Omega; V_h^l)$ is contained in the solution space $L^2(\Omega;V)$, 
the lifted semi-discrete problem fits into the abstract non-conforming finite element setting considered in the first lemma of Strang \cite{strang1973analysis}. 
We will now present these results in the context of our random Hilbert space setting.
\begin{lemma}[First lemma of Strang]
Let $u_h$ denote the solution to the lifted semi-discrete problem (\ref{eq:LiftedSemi}) and assume that
the bilinear form $a_h^l(\cdot,\cdot)$ is uniformly $L^2(\Omega;V_h^l)$-elliptic, i.e. for some $\alpha >0$
 \[
  a_h^l(\varphi_h, \varphi_h) \geq \alpha \| \varphi\|_{L^2(\Omega;V)}^2
 \]
for all $\varphi \in L^2(\Omega;V_h^l)$ and $h\in (0,h_0)$. Then there exists a constant $C>0$ independent of $h$ such that
\begin{align}
\label{eq:strang}
 \|u - u_h\|_{L^2(\Omega;V)} 
 &\lesssim \inf\limits_{\varphi_h \in L^2(\Omega; V_h^l)}
 \left( 
 \| u - \varphi_h\|_{L^2(\Omega;V)}  
  + \sup\limits_{w_h \in L^2(\Omega;V_h^l)} \frac{|a(\varphi_h, w_h) - a_h^l(\varphi_h, w_h)|}{\|w_h\|_{L^2(\Omega;V)}} \right)\\
  &+ \sup\limits_{w_h \in L^2(\Omega; V_h^l)} \frac{|l(w_h) - l_h^l(w_h)|}{\|w_h\|_{L^2(\Omega;V)}}. \nonumber
\end{align}
\end{lemma}

\begin{theorem}[Error estimates]\label{errorest}
Let $u$ denote the solution of the continuous problem (\ref{absmean}) and assume that it is sufficiently regular $u\in L^2(\Omega; Z_0)$
and let $U_h$ be the discrete solution of (\ref{abssemi}) with lift $u_h=\Lambda_h U_h. $
Then with the assumptions listed in section \ref{section:assumpList} satisfied, there exists a constant $c>0$ such that for all $h\in(0,h_0)$ we have the error estimate
\begin{equation}
\|u - u_h \|_{L^2(\Omega; H)} + h \|u - u_h \|_{L^2(\Omega; V) } \leq ch^{2} \|u\|_{L^2(\Omega; Z_0)}.
\end{equation}
\end{theorem}

\begin{proof}
It follows from the uniform ellipticity assumption (\ref{assumption:Coer1})
on the bilinear form $a_h(\cdot,\cdot)$ and the norm equivalence of the lifting map, 
that for any $\varphi_h = \Lambda_h \phi_h \in L^2(\Omega;V_h^l)$ we have
\[
 a_h^l(\varphi_h, \varphi_h) = a_h(\phi_h, \phi_h) \geq c \|\phi_h\|_{L^{2}(\Omega;V_h)}^2 \geq c \| \varphi_h\|_{L^2(\Omega;V)}^2.
\]
Therefore the bilinear form $a_h^l(\cdot,\cdot)$ is uniformly coercive and thus we can apply the first lemma of Strang. 
Substituting $\varphi_h = I_h u$ into the estimate (\ref{eq:strang}) and inserting the consistency bounds (P1), (P2) gives
\[
 \|u - u_h\|_{L^2(\Omega;V)} \lesssim \|u - I_h u \|_{L^2(\Omega;V)} + h \| I_h\|_{L^2(\Omega;V)} + h^2.
\]
Hence with the interpolation estimate (I1) applied to $u\in L^2(\Omega; Z_0)$ we obtain
\begin{equation}
\label{eq:lowerErrorEst}
 \|u - u_h\|_{L^2(\Omega;V)} \lesssim  h \|u \|_{L^2(\Omega;Z_0)}. 
\end{equation}
For the $L^2(\Omega; H)-$estimate, we use a standard duality argument. Given $g \in L^2(\Omega;H)$ and an arbitrary $w_h \in L^2(\Omega; V_h^l)$
we have
\begin{align*}
 (u-u_h, g)_{L^2(\Omega;H)} &= a(u-u_h, w(g) - w_h) + a(u-u_h, w_h)\\
 &= a(u-u_h, w(g) - w_h)
 + l(w_h) - l_h^l(w_h) 
 -\left( a(u_h,w_h)- a_h^l(u_h,w_h) \right)\\
 &= \RN{1} + \RN{2} + \RN{3}.
\end{align*}
Choosing $w_h = I_h w(g)$ and applying the interpolation estimate (I1) to the solution of the dual problem which is assumed (R1) to be 
sufficiently regular $w(g) \in L^2(\Omega; Z_0)$ gives
\begin{align*}
 |\RN{1}| & \lesssim \|u-u_h\|_{L^2(\Omega; V)} \| w(g) - I_h w(g)\|_{L^2(\Omega;V)} \\
 &\lesssim h^2 \|u\|_{L^2(\Omega;Z_0)}  \| w(g) \|_{L^2(\Omega;Z_0)}\\
 &\lesssim h^2 \|u\|_{L^2(\Omega; Z_0)} \|g \|_{L^2(\Omega; H)}.
\end{align*}
We bound the consistency error in the second term with (P2) giving 
\begin{align*}
 |\RN{2}| &\lesssim h^2 \| I_h w(g) \|_{L^2(\Omega;V) } 
 \lesssim h^2 \|w(g) \|_{L^2(\Omega;Z_0)} 
 \lesssim h^2 \|g\|_{L^2(\Omega; H)}.
\end{align*}
To obtain a bound of order $h^2$ for the third term, we begin by rewriting it as follows
\begin{align*}
 \RN{3} &= a(u_h, w(g) - I_h w(g) ) - a_h^l( u_h, w(g) - I_hw(g) )\\
 &+ a(u-u_h, w(g)) - a_h^l(u-u_h,w(g))\\
 &- \left( a(u,w(g)) - a_h^l(u,w(g)) \right).
\end{align*}
Now we are able to apply the estimate (P3) to the last term since both $u, w(g) \in L^2(\Omega;Z_0)$ and can then follow a similar argument
as to the previous cases for the first two terms which leads to 
\begin{align*}
 |\RN{3}| &\lesssim 
 h \| u_h\|_{L^2(\Omega;V)} \|w(g) - I_h w(g) \|_{L^2(\Omega; V)}
 +h \| u - u_h \|_{L^2(\Omega;V)} \| w(g) \|_{L^2(\Omega; V)}\\
 &+ h^2 \|u\|_{L^2(\Omega; Z_0)} \| w(g) \|_{L^2(\Omega; Z_0)}\\
 &\lesssim h^2 \| w(g)\|_{L^2(\Omega; Z_0)}
 + h^2 \|u \|_{L^2(\Omega;Z_0)} \| w(g) \|_{L^2(\Omega; Z_0)}\\
 &\lesssim h^2 \|u \|_{L^2(\Omega; Z_0)} \|g\|_{L^2(\Omega; H)}. 
\end{align*}
Combining the results gives the stated result
\[
 \| u - u_h\|_{L^2(\Omega;H)} = \sup\limits_{g \in L^2(\Omega; H) \setminus \{0\}} \frac{(u-u_h,g)_{L^2(\Omega;H)}}{\|g\|_{L^2(\Omega;H)}} 
 \lesssim h^2 \|u \|_{L^2(\Omega; Z_0)}.
\]

\end{proof}
We conclude our abstract error analysis by combining our finite element discretisation with the Monte-Carlo method to estimate our quantity of interest, the mean solution $E[u]$. 
Recall, that for an arbitrary Hilbert space $\mathcal{H}$, the Monte-Carlo estimator of the expectation 
of a random variable $Y \in L^2(\Omega; \mathcal{H})$ is a $\mathcal{H}$-valued random variable $E_M[Y] : \otimes_{i=1}^M \Omega \rightarrow \mathcal{H}$ defined by 
$$
E_M[Y] = \frac{1}{M} \sum\limits_{i=1}^M \hat{Y}_i
$$
where $M\in \mathbb{N}$ is the chosen number of samples taken and $\hat{Y}_i$ are independent identically distributed copies of the random variable $Y$.
Furthermore, we have the following well-known convergence result, see \cite{lord2014introduction}.
\begin{lemma}[Monte-Carlo convergence rate]\label{mcest}
For a given $M\in \mathbb{N}$ and a $\mathcal{H}$-valued random variable $Y\in L^2(\Omega; \mathcal{H})$, the Monte-Carlo estimator satisfies the convergence rate
\begin{equation}
\|E[Y] - E_M[Y] \|_{L^2(\Omega^M; \mathcal{H})} \leq \frac{1}{\sqrt{M}}\|Y\|_{L^2(\Omega;\mathcal{H})}.
\end{equation}
\end{lemma}
Therefore, if we consider the error between the mean solution $\mathbb{E}[u]$ and our discrete approximation $\mathbb{E}[u_h]$ in the $L^2(\Omega^M;H)$ norm,
and decompose it into the error arising from the finite element discretisation and the statistical error for the Monte-Carlo approximation, 
we obtain the following bound
\begin{align*}
 \| E[u] - E_M[u_h] \|_{L^2(\Omega^M; H)}
 &\leq \| E[u] - E[u_h] \|_{L^2(\Omega^M; H)} + \| E[u_h] - E_M[u_h] \|_{L^2(\Omega^M; H)}\\
 &\leq \|u -u_h \|_{L^2(\Omega; H)} + \frac{1}{\sqrt{M}} \| u_h\|_{L^2(\Omega;H)}
 \lesssim    h^{2}+ \frac{1}{\sqrt{M}}
\end{align*}
A similar argument in the $L^2(\Omega; V)$ leads to the following convergence rates.
\begin{theorem}\label{theorem:ConvergenceRates}
Let all the conditions from Theorem \ref{errorest} be satisfied. Then we have the following error estimates
\begin{align}
    \| E[u] - E_M[u_h] \|_{L^2(\Omega^M; H)} &\lesssim h^{2} + \frac{1}{\sqrt{M}}\\
        \| E[u] - E_M[u_h] \|_{L^2(\Omega^M; V)} &\lesssim h + \frac{1}{\sqrt{M}}.
        \label{MC2}
\end{align}
\end{theorem}

\section{Discretisation of the reformulated elliptic PDEs on their expected domains}
In this section, we apply the results from the abstract theory to 
two finite element discretisation schemes for the reformulations of the two model elliptic equations. 
In each case, we will verify that all the listed assumptions in abstract setting are satisfied hence giving the stated convergence rate.

\subsection{The elliptic equation on a random surface}
To discretise the reformulation of the elliptic equation
\[
 - \Delta_{\Gamma(\omega)} u(\omega) + u(\omega) = f(\omega) \quad \text{on } \Gamma(\omega)
\]
on the expected domain, we propose a semi-discrete scheme using linear Lagrangian surface finite elements \cite{dziuk2013finite}.
 Our computational domain $\Gamma_h$ approximating the smooth expected hypersurface $\Gamma_0$
 will be a polyhedral surface
$$\Gamma_h = \bigcup\limits_{T \in \mathcal{T}_h} T \subset U_{\delta}$$
consisting of finitely many non-degenerate triangles whose vertices are taken to lie on the surface $\Gamma_0$ and have the maximum diameter bounded above by $h>0$. 
The triangulation will be assumed to be shape regular and quasi-uniform, in the sense that the in-ball radius of each element
is uniformly bounded below by $ch$, for some constant $c>0$. 
In order to lift functions between the continuous and discrete surface, we shall assume that the projective mapping $a:\Gamma_h\rightarrow \Gamma_0$ decribed in (\ref{eq:aProj}) is
bijective and define the lift and inverse lift of functions $f$ and $g$ given over $\Gamma_h$ and $\Gamma_0$ respectively by
\begin{align}\label{eq:liftFunctions}
f^l(a) = f(x(a)) \quad 
g^{-l}(x) = g(a(x)) \quad \text{for } a \in \Gamma_0, x \in \Gamma_h,
\end{align}
where $x(a)$ denotes the inverse of the projection mapping $a$. 
We introduce the linear finite element space on $\Gamma_h$ 
\begin{equation}
    S_{h} = \{ \phi_h \in C^0(\Gamma_h) \, | \, \phi_h|_{T} \in \mathbb{P}_1(T), T \in \mathcal{T}_h\} 
\end{equation}
and define the lifted finite element space by
\begin{equation}
    S_{h}^l = \{ \varphi_h \in C^0(\Gamma_0)  \, | \, \varphi_h = \phi_h^l, \text{ for some } \phi_h \in S_{h}\}.
\end{equation}
The finite element discretisation of the mean-weak formulation reads as follows.
\begin{problem}[Semi-discrete scheme]
Find $U_h \in L^2(\Omega;S_{h})$ such that
\begin{equation}\label{p4}
\int_{\Omega} \int_{\Gamma_h} \mathcal{D}_{\Gamma_0}^{-l}(\omega) \nabla_{\Gamma_h} U_h(\omega) \cdot \nabla_{\Gamma_h} \phi_h(\omega) 
+  U_h(\omega) \phi_h(\omega) \sqrt{g_{\Gamma_0}^{-l}(\omega)}= \int_{\Omega} \int_{\Gamma_h} f^{-l}(\omega) \phi_h(\omega) \sqrt{g_{\Gamma_0}^{-l}(\omega)}
\end{equation}
for every $\phi_h \in L^2(\Omega; S_{h}).$
\end{problem}
In the context of the abstract framework, the finite dimensional space $\mathcal{V}_h$ is taken to be the finite element space $S_{h}$ and 
the Hilbert spaces $V_h, H_h$ are given by $H^1(\Gamma_h)$ and $L^2(\Gamma_h)$.
Furthermore, the abstract sample-dependent discrete bilinear form $\tilde{a}_h(\omega;\cdot, \cdot): H^1(\Gamma_h)\times H^1(\Gamma_h) \rightarrow \re$ and 
linear functional $\tilde{l}(\omega; \cdot): L^2(\Gamma_h) \rightarrow \re$ are given by
 \begin{align*}
  \tilde{a}_h(\omega; \chi_h, \phi_h ) &=
 \int_{\Gamma_h} \mathcal{D}_{\Gamma_0}^{-l}(\omega) \nabla_{\Gamma_h} \chi_h \cdot \nabla_{\Gamma_h} \phi_h +  \chi_h \phi_h \sqrt{g_{\Gamma_0}^{-l}(\omega)}\\
  \tilde{l}_h(\omega; \chi_h) &= \int_{\Gamma_h}  f^{-l}(\omega)\chi_h \sqrt{g_{\Gamma_0}^{-l}(\omega)}.
\end{align*}
With the uniform bounds on the random coefficients (\ref{eq:uniformBoundSurfCoef1}), (\ref{eq:uniformBoundSurfCoef2}), 
we deduce that $\tilde{a}_h(\omega:\cdot,\cdot)$ is uniformly $L^2(\Omega; H^1(\Gamma_0))$-elliptic and bounded,
and additionally $\tilde{l}(\omega; \cdot)$ is uniformly bounded as presumed in (\ref{assumption:Coer1} - \ref{assumption:Coer3}), 
and hence obtain existence and uniqueness of a semi-discrete solution to (\ref{p4}). 
We continue by checking the stated assumptions in the abstract error analysis.
 In particular, we begin with the norm equivalence (L1),(L2) of the lifting map $\Lambda_h:\mathcal{V}_h \rightarrow V$ given by $\Lambda_h \chi_h = \chi_h^l$.
 A proof of these estimates can be found in \cite[Lemma 4.2]{dziuk2013finite}.
\begin{lemma}[Equivalence in norms of lifts]\label{lifts}
There exists constants $c_1,c_2>0$ independent of $h$ such that for any $\chi_h \in S_{h}$ with lift $\chi_h^l \in S_{h}^l$ we have
\begin{align*}
    c_1 \|\chi_h\|_{L^2(\Gamma_h)} &\leq \|\chi_h^l\|_{L^2(\Gamma_0)} \leq c_2 \|\chi_h\|_{L^2(\Gamma_h)},\\
      c_1 \|\nabla_{\Gamma_h}\chi_h\|_{L^2(\Gamma_h)} &\leq \|\nabla_{\Gamma_0}\chi_h^l\|_{L^2(\Gamma_0)} \leq c_2 \|\nabla_{\Gamma_h}\chi_h\|_{L^2(\Gamma_h)}.
\end{align*}
\end{lemma}
For the interpolation assumption (I1), we set the Hilbert space $Z_0$ consisting of functions of higher regularity to be $H^2(\Gamma_0)$.
It follows from the Sobolev embedding that $H^2(\Gamma_0) \subset C^0(\Gamma_0)$ for $n \leq 3$ and therefore we can introduce the interpolation operator
 $I_h : H^2(\Gamma_0) \rightarrow S_{h}^l$ defined by
\begin{equation}
 I_h \eta = \left(\hat{I}_{h} \eta^{-l}\right)^{l}
\end{equation}
where $\hat{I}_h : C^0(\Gamma_h) \rightarrow S_{h}$ denotes
the standard Lagrangian interpolatant defined element-wise on $\Gamma_h$.
The following estimate was proved in \cite[Lemma 4.3]{dziuk2013finite}. 
\begin{lemma}[Interpolation estimate]\label{interpolationest}
Given any $\eta\in H^2(\Gamma_0)$, there exists a constant $c>0$ independent of $h$ such that 
\begin{equation}
\|\eta - I_h \eta \|_{L^2(\Gamma_0)} + h \| \nabla_{\Gamma_0}( \eta - I_h \eta) \|_{L^2(\Gamma_0)} \leq c h^{2} \| \eta\|_{H^2(\Gamma_0)}.
\end{equation}
\end{lemma}
To derive the assumed bounds (P1),(P2) and (P3) on the approximation of the discrete bilinear forms, 
we first need a preliminary result on the order of approximation of the geometry, see \cite[Lemma 4.1]{dziuk2013finite}.
\begin{lemma}[Geometric error bounds]\label{geometricerror}
Let $\delta_{h}^{\Gamma_0}$ denote the surface element corresponding to the transformation from $\Gamma_0$ to $\Gamma_h$ under the lifting map 
$d\sigma(a(x)) = \delta_{h}(x) d\sigma_{h}(x)$ and define
\begin{equation}\label{R_h}
 R_{h}^{\Gamma_0}(\omega) = \frac{1}{\delta_{h}^{\Gamma_0}} \left(\mathcal{D}_{\Gamma_0}^{-l}(\omega)\right)^{-1}\mathcal{P}_{\Gamma_0} 
 (I - d^{\Gamma_0} \mathcal{H}^{\Gamma_0}) \mathcal{P}_{h} \mathcal{D}_{\Gamma_0}^{-l}(\omega) \mathcal{P}_{h} (I - d^{\Gamma_0} \mathcal{H}^{\Gamma_0}),  
\end{equation}
where $\mathcal{P}_{h}:= I - \nu_{h}\otimes\nu_{h}$ is the projection operator mapping onto the tangent space of the discrete surface $\Gamma_h$
defined element-wise.
Then we have the estimates
\begin{align}
\label{eq:surfGeoEstD}
 \|d^{\Gamma_0}\|_{L^{\infty}(\Gamma_h)} &\leq c h^{2}\\
 \label{eq:surfGeoEstJ}
 \|1 - \delta_{h}^{\Gamma_0} \|_{L^{\infty}(\Gamma_h)} &\leq c h^{2}\\
 \label{eq:surfGeoEstR}
\|(I- R_{h}^{\Gamma_0}(\omega)) \mathcal{P}_{\Gamma_0}\|_{L^{\infty}(\Gamma_h)} &\leq ch^{2}.    
\end{align}
\end{lemma}
We can now bound the consistency error as follows. 
\begin{lemma}[Consistency error]
Given any $(W_h, \phi_h) \in L^2(\Omega; S_h) \times L^2(\Omega; S_h)$ with lifts \newline $(w_h, \varphi_h) \in L^2(\Omega;S_h^l)\times L^2(\Omega; S_h^l)$, 
we have 
\begin{align}\label{eq:ConsistLSurf}
|l(\varphi_h) - l_h(\phi_h) | &\leq c h^2 \|\varphi_h\|_{L^2(\Omega;L^2(\Gamma_0))}\\
\label{a_bound}
|a(w_h, \varphi_h ) - a_h(W_h, \phi_h ) | &\leq c h^{2} \| w_h\|_{L^2(\Omega;H^1(\Gamma_0))} \|\varphi_h\|_{L^2(\Omega;H^1(\Gamma_0))}.
\end{align}
\end{lemma}

\begin{proof}
Lifting the discrete integral in the linear functional $l_h(\cdot)$ onto the smooth surface $\Gamma_0$ with the projective mapping $a(\cdot)$ leads to 
\begin{align*}
    l(\varphi_h) - l_h(\phi_h)
    = 
    \int_{\Omega}\int_{\Gamma_0} \left( 1- \frac{1}{\delta_h^{\Gamma_0}}\right) f(\omega) \varphi_h(\omega) \sqrt{g_{\Gamma_0}(\omega)}.
\end{align*}
Hence with the uniform bound (\ref{eq:uniformBoundSurfCoef2}) on the random coefficient $\sqrt{g_{\Gamma_0}(\omega)}$ and the order $h^2$ approximation of 
the geometric pertubation (\ref{eq:surfGeoEstJ}), we obtain the estimate (\ref{eq:ConsistLSurf}).  
For (\ref{a_bound}), we begin by applying the chain rule to lift $W_h(\omega,x) = w_h(\omega, a(x))$
\begin{align*}
    \nabla_{\Gamma_h} W_h(\omega, x) = \mathcal{P}_{h}(x) (I - d^{\Gamma_0}(x) \mathcal{H}(x) )\mathcal{P}_{\Gamma_0}(x) \nabla_{\Gamma_0}w_h(\omega, a(x)).
\end{align*}
Suppressing the parameter $x$, we deduce
\begin{align*}
\mathcal{D}_{\Gamma_0}^{-l}(\omega) \nabla_{\Gamma_h}W_h(\omega) \cdot \nabla_{\Gamma_h} \phi_h(\omega) 
&= 
\mathcal{D}_{\Gamma_0}^{-l}(\omega) \mathcal{P}_{h}(I - d^{\Gamma_0} \mathcal{H}) \mathcal{P}_{\Gamma_0} \nabla_{\Gamma_0} w_h(\omega, a) \cdot 
\mathcal{P}_{h}(I - d^{\Gamma_0} \mathcal{H}) \mathcal{P}_{\Gamma_0} \nabla_{\Gamma_0} \varphi_h(\omega,a)\\
&= \mathcal{P}_{\Gamma_0} (I - d^{\Gamma_0} \mathcal{H}) \mathcal{P}_{h} \mathcal{D}_{\Gamma_0}^{-l}(\omega)
\mathcal{P}_{h}(I- d^{\Gamma_0}\mathcal{H}) \mathcal{P}_{\Gamma_0} \nabla_{\Gamma_0} w_h(\omega, a) \cdot \nabla_{\Gamma_0}\varphi_h(\omega,a)\\
&= \delta_{h}^{\Gamma_0} \mathcal{D}_{\Gamma_0}^{-l}(\omega) R_{h}^{\Gamma_0}(\omega)\nabla_{\Gamma_0}w_h(\omega) \cdot \nabla_{\Gamma_0} \varphi_h(\omega).
\end{align*}
Therefore, we can express the pertubation error in the approximation of the bilinear form $a(\cdot,\cdot)$ by
\begin{align*}
   a(w_h, \varphi_h) - a_h(W_h, \phi_h)  &= \int_{\Omega} \int_{\Gamma_0} \mathcal{D}_{\Gamma_0}(\omega) \left( \mathcal{P}_{\Gamma_0} - R_h^{\Gamma_0,l}(\omega) \right)
    \nabla_{\Gamma_0} w_h(\omega) \cdot \nabla_{\Gamma_0} \varphi_h(\omega) \\
    &+
    \int_{\Omega} \int_{\Gamma_0} \left( 1 - \frac{1}{\delta_h^{\Gamma_0,l}}\right) w_h(\omega) \varphi(\omega) \sqrt{g_{\Gamma_0}(\omega)}
\end{align*}
and hence with the uniform bounds (\ref{eq:uniformBoundSurfCoef1}), (\ref{eq:uniformBoundSurfCoef2}) on the random coefficients 
and the geometric estimates  (\ref{eq:surfGeoEstJ}), (\ref{eq:surfGeoEstR}) we obtain (\ref{a_bound}). 
\end{proof}
For the regularity assumption (R1) on the associated dual problem
\[
 a(\varphi, w(g) ) = (g, \varphi)_{L^2(\Omega;L^2(\Gamma_0))} \quad \text{for all } \varphi \in L^2(\Omega; H^1(\Gamma_0)),
\]
which due the symmetry of $\mathcal{D}_{\Gamma_0}$ and thus of $a(\cdot,\cdot)$,
is precisely the mean-weak formulation, we have the results presented in Theorem \ref{theorem:RegularitySurfMean}.

\subsection{The coupled elliptic system}
We next apply the results from the abstract framework to the second model problem of the coupled 
elliptic system 
\begin{align*}
 -\Delta u(\omega) + u(\omega) &= f(\omega) \quad \text{in } D(\omega)\\
 \alpha u(\omega) - \beta v(\omega) + \frac{\partial u}{\partial \nu_{\Gamma}}(\omega) &= 0 \quad \text{on } \Gamma(\omega)\\
 -\Delta_{\Gamma} v(\omega) + v(\omega) +  \frac{\partial u}{\partial \nu_{\Gamma}}(\omega) &= f_{\Gamma}(\omega) \quad \text{on } \Gamma(\omega)
\end{align*}
on a random bulk-surface. 
Our proposed finite element discretisation of the system reformulated on the expected domain and the subsequent analysis will be based on
the approach presented in \cite{elliott2012finite}.
For the computational domain, we approximate the open bulk $D_0\subset \re^{n+1}$ by a polyhedral domain
\[
 D_h = \bigcup\limits_{K \in \mathcal{T}_h} K
\]
consisting of closed $(n+1)-$simplices with maximum diameter uniformly bounded above by positive constant $h>0$ and will assume that the triangulation
$\mathcal{T}_h$ is quasi-uniform.
We denote the induced discrete surface $\Gamma_h = \partial D_h$ and the associated triangulation by
\[
 \Gamma_h = \bigcup\limits_{T \in \mathcal{T}_h} T
\]
and impose the same assumptions on $\mathcal{T}_h$ as were listed in the previous example.
A piece-wise diffeomorphic mapping $G_h: D_h\rightarrow D_0$
from the discrete bulk to the continuous can be constructed 
by fixing the interior simplices (simplices with at most one vertex on the boundary $\Gamma_0$)
and using the projective mapping $a^{\Gamma_0}(\cdot)$ to 
define a diffeomorphism 
$\Lambda_{h,k}:K \rightarrow K^e$
between the boundary simplices $K$ (simplices with at least two vertices on $\Gamma_0$) and the exact 
curved simplices $K^e$, 
\begin{equation}
 G_h|_K =
 \begin{cases}
  \Lambda_{h,K} \quad K \text{ boundary simplex}\\
  id|_{K} \quad K \text{ interior simplex.}
 \end{cases}
\end{equation}
Details on the precise form of $\Lambda_{h,K}$ can be found in \cite{elliott2012finite}. 
We are therefore able to define lifts and inverse lifts of functions on the bulk domain 
by
\begin{align}
\label{lift:bulka}
 \varphi_h^l(x) &= \varphi_h( G_h^{-1}(x)) \quad x \in D_0\\
 \varphi^{-l}(x) &= \varphi( G_h(x)) \quad x \in D_h.
\end{align}
Note that, the diffeomorphism $\Lambda_{h,K}$ is chosen such that the mapping $G_h$ coincides with 
the projective mapping 
\begin{equation}\label{eq:G_h}
 G_h(x) = a^{\Gamma_0}(x) \quad x \in \partial D_h
\end{equation}
on the boundary of the discrete bulk and hence the bulk lift agrees with the surface lifting map described 
in (\ref{eq:liftFunctions}) on $\partial D_h$.
For convenience, we will denote the sub-triangulation consisting of all boundary simplices by 
$$\mathcal{B}_h = \{ K \in \mathcal{T}_h \, | \, K \text{ is a boundary simplex} \}$$
 and define the corresponding sets 
\begin{equation}
\label{eq:Bdef}
 B_h 
 = \bigcup\limits_{K \in\mathcal{B}_h}K
 \quad
 B_h^l = \bigcup\limits_{K \in B_h} K^e
\end{equation}
where the lifting maps $G_h, G_h^{-1}$ differ from the identity mapping.
We introduce the linear finite element spaces on the discrete bulk and discrete surface by 
\begin{align}
 V_h &= \{ \phi_h \in C^0(D_h) \, | \, \phi_h|_{K} \in P^1(K)\text{ for all } K \in \mathcal{T}_h\}\\
 S_h &= \{ \zeta_h \in C^0(\Gamma_h) \, | \, {\zeta_h}|_{T} \in P^1(T) \text{ for all } T \in \mathcal{\check{T}}_h \} 
\end{align}
and denote the corresponding lifted finite element spaces by
\begin{equation}
 V_h^l = \{ \varphi_h = \phi_h^l \, | \, \phi_h \in V_h\} \quad
 S_h^l = \{ \xi_h = \zeta_h^l \, | \, \zeta_h \in S_h\}.
\end{equation}
An important feature of our finite element spaces is that the trace of a function $\phi_h \in V_h$ belongs to $S_h$ 
and similarly the trace of $\varphi_h \in V_h^l$ belongs to $S_h^l$ as a result of (\ref{eq:G_h}).
The finite element discretisation of the mean-weak formulation then reads as follows.
\begin{problem}[Semi-discrete problem]\label{Semi-discrete problem} 
Find a pair $(U_h, V_h) \in L^2(\Omega; V_h \times S_h)$ such that
\begin{align*}
 \alpha \int_{\Omega}\int_{D_h} \mathcal{D}^{-l}(\omega) \nabla U_h(\omega) &\cdot \nabla \phi_h(\omega) + U_h(\omega) \phi_h(\omega) \sqrt{g^{-l}(\omega)}\\
 + \beta \int_{\Omega} \int_{\Gamma_h} \mathcal{D}_{\Gamma_0}^{-l}(\omega) \nabla_{\Gamma_h}V_h(\omega) &\cdot \nabla_{\Gamma_h}\zeta_h(\omega)
 + V_h(\omega) \zeta_h(\omega) \sqrt{g_{\Gamma_0}^{-l}(\omega)}\\
 \int_{\Omega} \int_{\Gamma_h} \left( \alpha U_h(\omega) - \beta V_h(\omega) \right) &\left( \alpha \phi_h(\omega) - \beta \zeta_h(\omega) \right) \sqrt{g_{\Gamma_0}^{-l}(\omega)}\\
 &= 
 \alpha \int_{\Omega}\int_{D_h} f^{-l}(\omega) \phi_h(\omega) \sqrt{g^{-l}(\omega)}
 + \beta \int_{\Omega} \int_{\Gamma_h} f_{\Gamma_0}^{-l}(\omega) \zeta_h(\omega)\sqrt{g_{\Gamma_0}^{-l}(\omega) }
\end{align*}
for every $(\phi_h, \zeta_h) \in L^2(\Omega; V_h \times S_h). $
\end{problem}
Here the abstract finite dimensional space  is $\mathcal{V}_h = V_h\times S_h$ and 
the Hilbert spaces $V_h,H_h$ are given by $H^1(D_0)\times H^1(\Gamma_0)$ and $L^2(D_0)\times L^2(\Gamma_0)$ respectively.
We denote the associated bilinear form  
and linear functional 
$$
a_h(\cdot, \cdot): L^2(\Omega; V_h \times S_h) \times L^2(\Omega; V_h\times S_h) \rightarrow \re 
\quad
l_h(\cdot): L^2(\Omega; V_h \times S_h) \rightarrow \re
$$
to be the respective left hand side and right hand side of the semi-discrete 
variational problem \ref{Semi-discrete problem}. By the uniform bounds on the random coefficients (\ref{eq:UniformBoundBulkCoef}), (\ref{eq:uniformBoundSurfCoef1}), we deduce the existence and uniqueness of a semi-discrete solution using a similar
argument to the continuous problem. 
We proceed in a similar manner and check that the assumptions of the abstract analysis are satisfied. 
The norm equivalence (L1), (L2) of the lifting mapping which in this setting 
$\Lambda_h : V_h\times S_h \rightarrow V_h^l\times S_h^l$ is given component-wise by 
\begin{equation}
\label{lift:coupled}
 \Lambda_h\left((\phi_h, \zeta_h) \right) = ( \phi_h^l, \zeta_h^l),
\end{equation}
follows from the estimates on the surface lifting map given Lemma \ref{lifts} in combination with the following bulk lifting norm equivalence 
 derived in \cite[Proposition 4.9]{elliott2012finite}.
\begin{lemma}[Bulk lift estimates]
 There exists constants $c_1,c_2>0$ independent of $h$, such that 
 for any $\phi_h: D_h\rightarrow \re$ with lift $\varphi_h = \phi_h^l:D_0\rightarrow \re$ we have
 \begin{align*}
  c_1 \| \phi_h\|_{L^2(D_h)} &\leq \| \varphi_h\|_{L^2(D_0)} \leq c_2 \| \phi_h\|_{L^2(D_h)}\\
  c_1 \| \phi_h\|_{H^1(D_h)} &\leq \| \varphi_h\|_{H^1(D_0)} \leq c_2 \| \phi_h\|_{H^1(D_h)}.
  \end{align*}
\end{lemma}
For the interpolation assumption (I1), we set the abstract function space $Z_0 = H^2(D_0)\times H^2(\Gamma_0)$ and define the interpolation operator component-wise
\begin{equation}
  I_h(\eta, \xi) = \left(( \tilde{I}_h \eta^{-l})^l , ( \tilde{I}_h \xi^{-l})^l \right)
\end{equation}
with $\tilde{I}_h$ denoting the standard Lagrangian intepolation operator
and have the following estimate .
\begin{lemma}[Interpolation estimate] 
 There exists a well-defined interpolation operator
 $$I_h: H^2(D_0)\times H^2(\Gamma_0) \rightarrow V_h^l\times S_h^l$$
 such that for any $(\eta, \xi)\in H^2(D_0)\times H^2(\Gamma_0)$ we have
 \begin{equation}
  \label{eq:InterpCoupled}
  \|(\eta, \xi) - I_h(\eta, \xi) \|_{L^2(D_0)\times L^2(\Gamma_0)}
  +
  h\|(\eta, \xi) - I_h(\eta, \xi) \|_{H^1(D_0)\times H^1(\Gamma_0)}
  \leq 
  ch^2 \|(\eta, \xi)\|_{H^2(D_0)\times H^2(\Gamma_0)}.
 \end{equation}
\end{lemma}

The next step will entail bounding the consistency error arising from the geometric approximation of the domain. 
Estimates for the surface pertubation have previously been given in Lemma \ref{geometricerror}. 
For the bulk approximation, we recall that the lifting mapping $G_h: D_h\rightarrow D_0$ is defined to be the identity on interior simplices and 
a $C^1-$diffeomorphism for simplices near the boundary. Therefore the corresponding bulk error will be comprised of two parts; the first part will be related to the smallness
of the neighbourhood around $\Gamma_0$ in which the lifted boundary simplices lie in and the second part 
is the associated geometric error of the boundary simplices approximating the corresponding exact curved simplex.
We begin with the latter and state geometric bulk estimates on the diffeomorphic mapping $G_h$, for which a proof of the bounds (\ref{eq:geoBulkEst1}) and (\ref{eq:geoBulkEst2})
can be found in \cite[Proposition 4.7]{elliott2012finite}.
\begin{lemma}[Geometric bulk estimates]
 Let $\delta_h^{D_0} = |det(\nabla G_h)|$ be the volume element corresponding to the transformation $G_h: D_h\rightarrow D_0$ and set
 \[
  R_h^{D_0}(\omega) = \frac{1}{\delta_h^{D_0}} \left( \mathcal{D}^{-l}(\omega) \right)^{-1}  \nabla G_h \mathcal{D}^{-l}(\omega) \nabla G_h^{\top}. 
 \]
Then we have the following estimates for a constant $c>0$ independent of $\omega$,
\begin{align}
\label{eq:geoBulkEst1}
 \|\nabla G_h - I \|_{L^{\infty}(D_h)} &\leq c h\\
 \label{eq:geoBulkEst2}
\| \delta_h^{D_0} - 1 \|_{L^{\infty}(D_h)} &\leq c h\\
\label{eq:geoBulkEst3}
\|R_h^{D_0}(\omega) - I \|_{L^{\infty}(D_h)} &\leq ch.
\end{align}
\end{lemma}
\begin{proof}
 The estimate (\ref{eq:geoBulkEst3}) follows from the observation
 \begin{align*}
  R_h^{D_0}(\omega) - I  
  &= \frac{1}{\delta_h^{D_0}} \left( \mathcal{D}^{-l}(\omega)\right)^{-1} \nabla G_h \mathcal{D}^{-l}(\omega) \left( \nabla G_h^{\top} - I \right)
  + \frac{1}{\delta_h^{D_0}} \left( \mathcal{D}^{-l}(\omega)\right)^{-1} \left( \nabla G_h - I\right) \mathcal{D}^{-l}(\omega)\\
  &+ \left( \frac{1}{\delta_h^{D_0}} -1 \right) I.
 \end{align*}
and the uniform bounds (\ref{eq:UniformBoundBulkCoef}) on the random coefficient $\mathcal{D}(\omega)$.  
\end{proof}
To obtain a bound on the open neighbourhood containing the boundary simplices, we have the subsequent narrow band inequality  \cite[Lemma 4.10]{elliott2012finite}.
\begin{lemma}[Narrow band trace inequality]
\label{lemma:NarrowBand}
Given any $\delta< \delta_{\Gamma_0}$, let $\mathcal{N}_{\delta}$ be a narrow band in the interior domain $D_0$ around the boundary $\Gamma_0$ 
defined by
\begin{equation}
\label{def:N}
 \mathcal{N}_{\delta} =\{ x \in D_0\, | \, -\delta < d(x) <0\}.
\end{equation}
 Then for any $\eta \in H^1(D_0)$ we have
\[
 \| \eta \|_{L^2(\mathcal{N}_{\delta})} \leq c \delta^{\frac{1}{2}} \|\eta \|_{H^1(D_0)}.
\]
\end{lemma}

The consistency error can now be bounded as follows.
\begin{lemma}[Consistency error]
Assume $f \in L^2(\Omega; H^1(D_0))$. Then for any 
$\phi_h, W_h \in L^2(\Omega;V_h)$ and $\zeta_h, X_h \in L^2(\Omega; S_h)$ with corresponding lifts
$\varphi_h, w_h $ and $\xi_h, \chi_h  $
we have 
\begin{align}
\label{eq:ConsistL}
 |l(\varphi_h, \xi_h) - l_h(\phi_h, \zeta_h)| &\leq ch^2 \| (f,f_{\Gamma_0})\|_{L^2(\Omega;H^1(D_0)\times L^2(\Gamma_0))}
 \|(\varphi_h, \xi_h)\|_{L^2(\Omega; H^1(D_0)\times H^1(\Gamma_0))}\\
 \label{eq:ConsistA1}
 |a\left( (\varphi_h, \xi_h), (w_h, \chi_h) \right) - a_h( &(\phi_h, \zeta_h),(W_h,X_h))|\\
 &\leq ch \|(\varphi_h, \zeta_h)\|_{L^2(\Omega; H^1(D_0)\times H^1(\Gamma_0))}
 \|(w_h, \chi_h) \|_{L^2(\Omega; H^1(D_0)\times H^1(\Gamma_0))}. \nonumber
\end{align}
Furthermore, 
for any $\varphi, w \in L^2(\Omega; H^2(D_0))$ and $\xi, \chi \in L^2(\Omega; H^2(\Gamma_0))$ 
with inverse lifts $\varphi^{-l}, w^{-l}$ and $\xi^{-l},  \chi^{-l}$
we have 
\begin{align}
 \label{eq:ConsistA2}
|a\left( (\varphi, \xi), (w, \chi) \right) &- a_h\left( (\varphi^{-l}, \xi^{-l}), ( w^{-l}, \chi^{-l}) \right)|\\
 &\leq c h^2 \| ( \varphi, \xi) \|_{L^2(\Omega; H^2(D_0)\times H^2(\Gamma_0))}
\|(w, \chi)\|_{L^2(\Omega; H^2(D_0)\times H^2(\Gamma_0))}. \nonumber
 \end{align}
\end{lemma}

\begin{proof}
 For the estimate (\ref{eq:ConsistL}), we begin by lifting the discrete integrals in $l_h(\cdot)$ onto their respective continuous counterparts
 recalling that the set of all boundary simplices $B_h$ is the region in which the diffeomorphic mapping $G_h$ differs from the identity and 
thus where $\delta_h^{D_0} = det(\nabla G_h) \neq 1$,
 \begin{align*}
  l(\varphi_h, \xi_h) &- l_h(\phi_h, \zeta_h) \\
  &= \alpha \int_{\Omega} \int_{D_0} \left( 1 - \frac{1}{\delta_h^{D_0,l}} \right) f(\omega) \varphi_h(\omega) \sqrt{g(\omega)}
  + \beta \int_{\Omega} \int_{\Gamma_0}  \left( 1 - \frac{1}{\delta_h^{\Gamma_0,l}} \right) f_{\Gamma_0}(\omega) \xi_h(\omega) \sqrt{g_{\Gamma_0}(\omega)}\\
  &= \alpha \int_{\Omega} \int_{B_h^l} \left( 1 - \frac{1}{\delta_h^{D_0,l}} \right) f(\omega) \varphi_h(\omega) \sqrt{g(\omega)}
  + \beta \int_{\Omega} \int_{\Gamma_0}  \left( 1 - \frac{1}{\delta_h^{\Gamma_0,l}} \right) f_{\Gamma_0}(\omega) \xi_h(\omega) \sqrt{g_{\Gamma_0}(\omega)}.
  \end{align*}
 Substituting the geometric bulk and surface estimates (\ref{eq:geoBulkEst2}), (\ref{eq:surfGeoEstJ}) 
with the uniform bounds on the random coefficients (\ref{eq:UniformBoundBulkCoef}),
(\ref{eq:uniformBoundSurfCoef1})
leads to 
\begin{align}
\nonumber
 |l(\varphi_h, \xi_h) - l_h(\phi_h, \zeta_h)|
\lesssim h \|f \|_{L^2(\Omega; L^2(B_h^l))} \| \varphi_h\|_{L^2(\Omega;L^2(B_h^l))} 
 + h^2  \|f_{\Gamma_0} \|_{L^2(\Omega; L^2(\Gamma_0))} \| \xi_h\|_{L^2(\Omega;L^2(\Gamma_0))}.
\end{align} 
To obtain a bound of order $h^2$ on the bulk term, we will now apply the narrow trace band inequality. 
We choose $\delta>0$ such that 
$0 < h < \delta < ch$ for some constant $c>0$, thus giving
\begin{equation}
 \|f\|_{L^2(\Omega; L^2(B_h^l))} \leq \| f \|_{L^2(\Omega; L^2(\mathcal{N}_{\delta}))} 
 \leq c \delta^{\frac{1}{2}}\| f \|_{L^2(\Omega; H^1(D_0))} \leq c h^{\frac{1}{2}} \| f \|_{L^2(\Omega; H^1(D_0))}.
\end{equation}
With a similar estimate on the test function $\varphi_h$, we obtain (\ref{eq:ConsistL}).
For (\ref{eq:ConsistA1}) and (\ref{eq:ConsistA2}), we apply the chain rule to the lifts 
$
 \varphi_h(\omega, G_h(x)) = \phi_h(\omega, x) 
$
and $w_h(\omega, G_h(x)) = W_h(\omega, x)$
to deduce
\begin{align*}
 \mathcal{D}^{-l}(\omega, x) \nabla \phi_h(\omega, x) \cdot \nabla W_h(\omega, x) 
 &= \mathcal{D}^{-1}(\omega,x) \nabla G_h^{\top}(x) \nabla \varphi_h(\omega, G_h(x)) \cdot \nabla G_h^{\top}(x) \nabla w_h(\omega, G_h(x))\\
 &= \nabla G_h(x) \mathcal{D}^{-l}(\omega,x) \nabla G_h^{\top}(x) \nabla \varphi_h(\omega, G_h(x)) \cdot \nabla w_h(\omega,G_h(x))\\
 &= \delta_h^{D_0}(x) \mathcal{D}^{-l}(\omega,x) R_h^{D_0}(\omega,x) \nabla \varphi_h(\omega,G_h(x)) \cdot \nabla w_h(\omega, G_h(x)).
\end{align*}
We can therefore express the perturbation error in our approximation of $a(\cdot, \cdot)$ as follows
\begin{align*}
 a( (\varphi_h,& \xi_h),(w_h, \chi_h) ) - a_h( (\phi_h, \zeta_h),(W_h,X_h)) \\
 &= \alpha \int_{\Omega} \int_{B_h^l} \mathcal{D}(\omega) \left( I- R_h^{D_0,l} (\omega) \right) \nabla \varphi_h(\omega) \cdot \nabla w_h(\omega) 
 + \left(1 - \frac{1}{\delta_h^{D_0,l}} \right) \varphi_h(\omega) w_h(\omega) \sqrt{g(\omega)}\\
 &+ \beta \int_{\Omega} \int_{\Gamma_0} \mathcal{D}_{\Gamma_0}(\omega) \left( \mathcal{P}_{\Gamma_0} - R_h^{\Gamma_0,l} \right) \nabla_{\Gamma_0} \xi_h(\omega)
 \cdot \nabla_{\Gamma_0} \chi_h(\omega)
 + \left( 1 - \frac{1}{\delta_h^{\Gamma_0,l}} \right) \xi_h(\omega) \chi_h(\omega) \sqrt{g_{\Gamma_0}(\omega) } \\
 &+ \int_{\Omega} \int_{\Gamma_0} \left( 1 - \frac{1}{\delta_h^{\Gamma_0,l}} \right)  \left( \alpha \varphi_h(\omega) - \beta \xi_h(\omega) \right)
 \left( \alpha w_h(\omega) - \beta \chi(\omega) \right) \sqrt{g_{\Gamma_0}(\omega)}. \\
\end{align*}
Here we have again used the fact that the diffeomorphic mapping $G_h$ is the identity on interior simplices and 
consequently $\delta_h^{D_0} = 1$ and $R_h^{D_0} = I$ on $D_h \setminus B_h.$
We now apply the geometric estimates and bounds on the random coefficients to obtain
\begin{align*}
  |a( (\varphi_h, \xi_h),(w_h, \chi_h) ) - a_h( (\phi_h, \zeta_h),(W_h,X_h))|  
 &\lesssim h \| \varphi_h\|_{L^2(\Omega; H^1(B_h^l))} \|w_h\|_{L^2(\Omega; H^1(B_h^l))}\\
 +& h^2 \| \xi_h\|_{L^2(\Omega; H^1(\Gamma_0))}\| \chi_h\|_{L^2(\Omega; H^1(\Gamma_0))}\\
 +& h^2 \| \alpha \varphi_h - \beta \xi_h\|_{L^2(\Omega; L^2(\Gamma_0))} \| \alpha w_h - \beta \chi_h\|_{L^2(\Omega; L^2(\Gamma_0))}\\
\end{align*}
For the last term, we observe by the boundedness of the trace operator $\|f\|_{L^2(\Gamma_0)} \leq c_T \|f\|_{H^1(D_0)}$ that
\begin{align*}
  \| \alpha \varphi_h &- \beta \xi_h\|_{L^2(\Omega; L^2(\Gamma_0))} \| \alpha w_h - \beta \chi_h\|_{L^2(\Omega; L^2(\Gamma_0))}\\
  &\leq \left( \alpha c_T \| \varphi_h\|_{L^2(\Omega; H^1(D_0))} + \beta \| \xi_h\|_{L^2(\Omega; L^2(\Gamma_0))} \right)
 \left( \alpha c_T \| w_h\|_{L^2(\Omega;H^1(D_0))} + \beta \| \chi_h\|_{L^2(\Omega; L^2(\Gamma_0))} \right)\\
 &\lesssim\| ( \varphi_h, \xi_h) \|_{L^2(\Omega; H^1(D_0)\times L^2(\Gamma_0))}
 \|( w_h, \chi_h) \|_{L^2(\Omega; H^1(D_0)\times L^2(\Gamma_0))}.
\end{align*}
Examining the bulk term,
we see that we are unable to apply the narrow band inequality Lemma \ref{lemma:NarrowBand}, to the derivative of $\varphi_h(\omega)$ and  $w_h(\omega)$ 
since the functions only belong to the space  $V_h \subset H^1(D_0)$, resulting in the bound of order $h$ given in (\ref{eq:ConsistA1}). However, 
considering sufficiently regular functions $\varphi,w \in L^2(\Omega; H^2(D_0))$, we are able to employ Lemma \ref{lemma:NarrowBand} attaining 
the estimate of order $h^2$ given in (\ref{eq:ConsistA2}).
\end{proof}

The regularity assumption (R1) on the associated dual problem follows again from the symmetry of the bilinear for $a(\cdot,\cdot)$ and 
the previously derived regularity result given in Theorem \ref{theorem:RegCoupled}.
Hence all the assumptions of the abstract theory are satisfied and we have the stated convergence rate given in Theorem \ref{theorem:ConvergenceRates}.


\section{Numerical results}
In this section, we numerically verify the stated convergence rates of the two proposed finite element discretisations of 
the reformulated model elliptic problems.
 In both cases, the numerical scheme has been implemented in DUNE \cite{bastian2008generic,dedner2010generic}. 
\subsection{Random Surface}
As a model for the random surface $\Gamma(\omega)$, we consider a graph-like representation over the unit sphere $\Gamma_0 = S^2$ 
\begin{equation}\label{eq:GraphicalReps}
 \Gamma(\omega) = \{ x + h(\omega,x) \nu^{\Gamma_0}(x) \, | \, x \in \Gamma_0 \},
\end{equation}
 where the prescribed height function $h(\omega,\cdot):\Gamma_0 \rightarrow \re$, will take the form of a truncated spherical harmonic expansion 
\begin{equation}
 h(\omega,x) = \epsilon_{tol} \sum\limits_{m<6} \sum\limits_{|l|\leq m} \lambda_{l,m}(\omega) Y_l^m(\theta, \phi) \quad x = ( sin\theta cos\phi, sin\theta sin\phi ),
\end{equation}
with independent, uniformly distributed random coefficients $\lambda_{l,m}\sim U(-1,1)$.
Here $\epsilon_{tol}>0$ is a parameter controlling the maximum deviation of the fluctuating surface which in practice will be set to $\epsilon_{tol}=0.1$ 
and $Y_l^m$ denotes the spherical harmonic function of degree $l$ and order $m$, which correspond to the eigenvalues of the Laplace-Beltrami operator.
For further details on exact form of the spherical harmonics, we refer the reader to \cite{groemer1996geometric,atkinson2012spherical}.
Realisations of the random surface for different samples are given below in Figure \ref{fig:realSurf}.
\begin{figure}[!htbp]
\begin{center}
$
\begin{array}{ccc}
\includegraphics[width=1.5in]{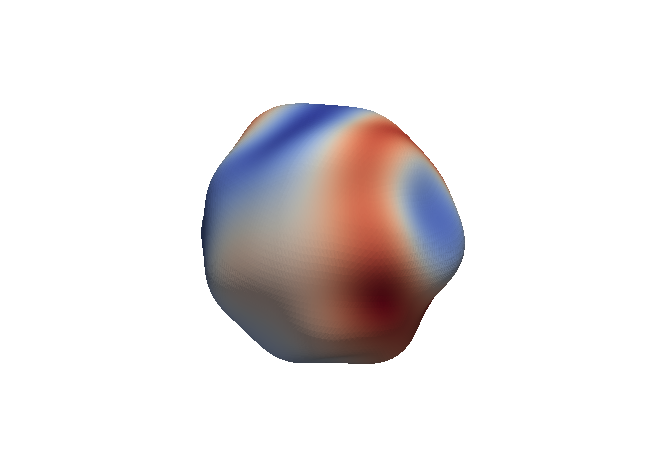} &
\includegraphics[width=1.5in]{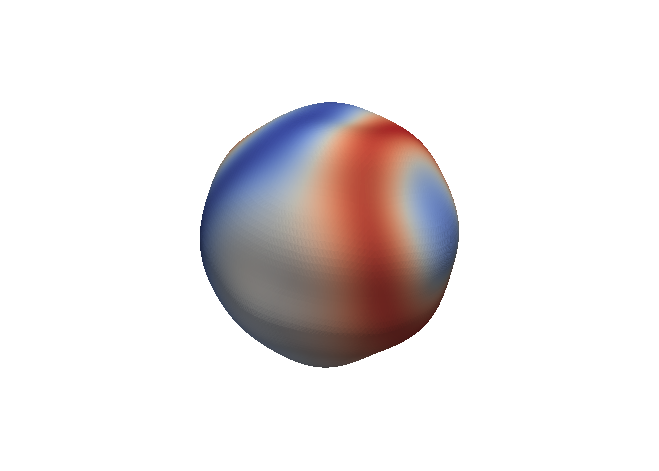} &
\includegraphics[width=1.5in]{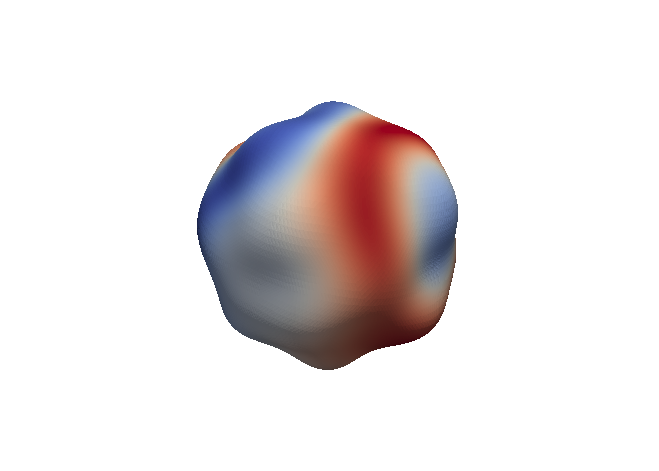} \\
\includegraphics[width=1.5in]{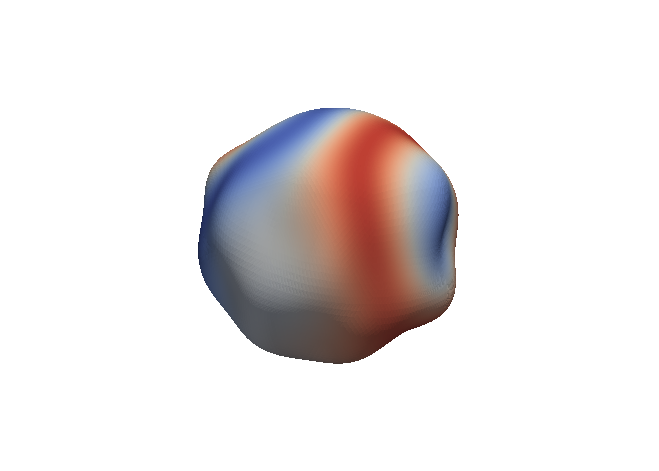} &
\includegraphics[width=1.5in]{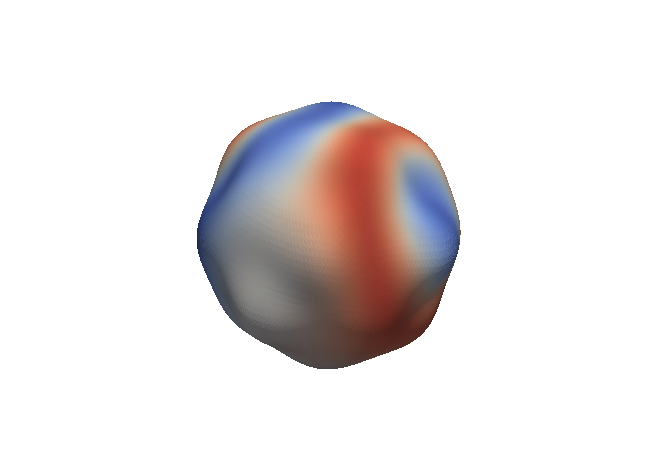} &
\includegraphics[width=1.5in]{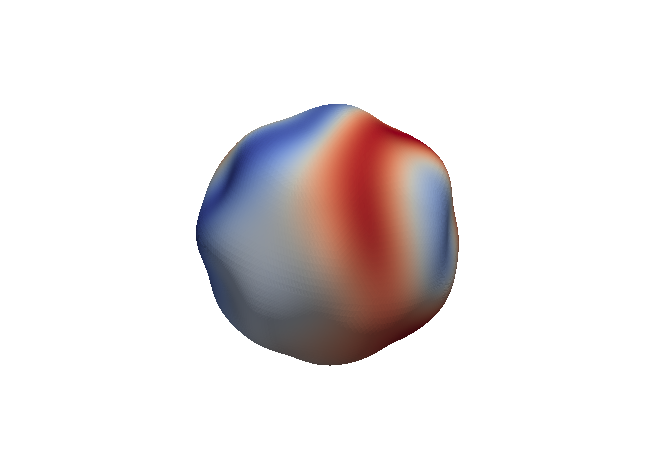} \\
\end{array}$
\end{center}
\caption{Realisations of the path-wise solution on the associated realisation of the random surface.}
\label{fig:realSurf}
\end{figure}

To numerical verify the convergence rate, we set the exact pull-back solution to be given by 
\begin{align*}
\hat{u}(\omega,x) &= sin( \pi (x^2 -1) y (z-1)) + \sigma_{tol} \nu_1(\omega) cos( \pi z (y+1)) +\sigma_{tol} \nu_2(\omega) sin(\pi( x+y) z^2)
\end{align*}
with $\nu_1, \nu_2 \sim U(-1,1)$ and $\sigma_{tol}>0$ a constant controlling the largest deviation of pathwise solution. 
This in turn determines the random data $\hat{f}$ given in the reformulated elliptic equation (\ref{eq:reformSurf}).
We observe the following errors for the approximation $\mathbb{E}[\hat{u}] - E_M[\hat{u}_h]$ in $L^2(\Omega^M; L^2(\Gamma_0))$ and $L^2(\Omega^M; H^1(\Gamma_0))$ 
and thus the stated convergence results.    
\begin{table}[h]
  \begin{minipage}{0.42\textwidth}
            \centering
            \begin{table}[H]
\centering
\renewcommand{\arraystretch}{1.3}
\begin{tabular}{|c c| c|c c|}
\hline
$h$	      	 & $M$	 	&  $E_{L^2(\Gamma_0)}$  &  $eoc(h)$ 	& $eoc(M)$      \\ \hline 
0.171499         &        1     &     0.776832          &       -       &          -    \\ \hline        
0.0877058        &        16    &      0.387486         &   1.03722     &    -0.250864  \\ \hline
0.0441081        &       256    &      0.106022         &   1.88556     &    -0.467444  \\ \hline          
0.0220863        &      4096    &     0.0267303         &   1.99202     &    -0.496955  \\ \hline 
\end{tabular}
\caption{ Error in $L^2(\Omega^M;L^2(\Gamma_0))$. } \label{table:surfaceL2}
\end{table}
        \end{minipage}
        \hspace{33pt}
        \begin{minipage}{0.42\textwidth}
            \centering
  \begin{table}[H]
\centering
\renewcommand{\arraystretch}{1.3}
\begin{tabular}{|c c| c|c c|}
\hline
$h$	      	 & $M$	 	&  $E_{H^1(\Gamma_0)}$  &  $eoc(h)$ 	&    $eoc(M)$      \\ \hline 
0.171499         &       64     &      4.89172          &      -        &          -       \\ \hline 		
0.0877058        &       256    &       3.68809         &   0.421176    &     -0.203734    \\ \hline 
0.0441081        &     1024     &      1.90402          &    0.961875   &      -0.476911   \\ \hline 
0.0220863        &      4096    &      0.961782         &    0.987348   &      -0.492633   \\ \hline  
\end{tabular}
\caption{Error in $L^2(\Omega^M;H^1(\Gamma_0))$.} \label{table:surfaceH1}
\end{table}
        \end{minipage}
    \end{table}
\subsection{Random bulk-surface}
For the coupled-elliptic system on a random bulk-surface, we adopt a similar approach to the random surface numerical example and
prescribe the curved boundary to the random bulk $D(\omega)$ which for simplicity is taken to lie in $\re^2$, as a graph
\begin{equation}
 \Gamma(\omega) = \{ x + h(\omega,x) \nu^{\Gamma_0}(x) \, | \, x \in S^1 \}
\end{equation}
over the unit circle. Here the random height function will given by a truncated Fourier series 
\[
h(\omega,x) = \sum\limits_{n =1}^6 \lambda_n(\omega) cos(n\theta) + \hat{\lambda}_n(\omega) sin(n\theta) \quad x = (cos(\theta),sin(\theta) ) \in S^1,
\]
with independent, uniformly distributed random coefficients $\lambda_n, \hat{\lambda}_n \sim U(-1,1)$.
We extend the given boundary process in the normal direction  into the interior with a sufficiently 
smooth blending function to form the stochastic domain mapping 
\begin{equation}\label{eq:stochasticMap}
\phi(x, \omega) = x + L_{\delta}(|x - a^{\Gamma_0}(x) |) h(a^{\Gamma_0}(x), \omega) \nu^{\Gamma_0}(a^{\Gamma_0}(x)) \quad x \in \overline{B_1(0)}.
\end{equation}
Here the precise form of the chosen blending function $L_{\delta}(\cdot): \re_{\geq 0} \rightarrow \re_{\geq 0}$ is given by
\[
 L_{\delta}(x) = \begin{cases}
                  exp\left( \frac{-x^2}{\delta^2 - x^2}\right) \quad &\text{if }x < \delta,\\
                  0 			           \quad &\text{if } x \geq \delta.		 
                 \end{cases}
\]
Realisations of the image of the reference domain mappped under the random domain mapping (\ref{eq:stochasticMap}) are provided in Figure \ref{fig:realCoupled}.
\begin{figure}[!htbp]
\begin{center}$
\begin{array}{cccc}
\includegraphics[width=1.5in]{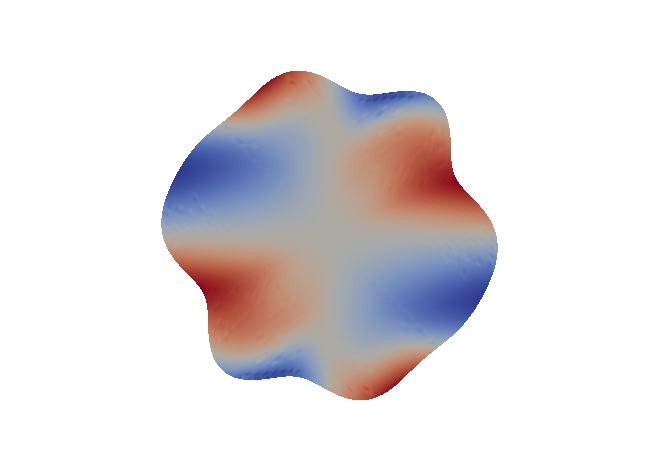} &
\includegraphics[width=1.5in]{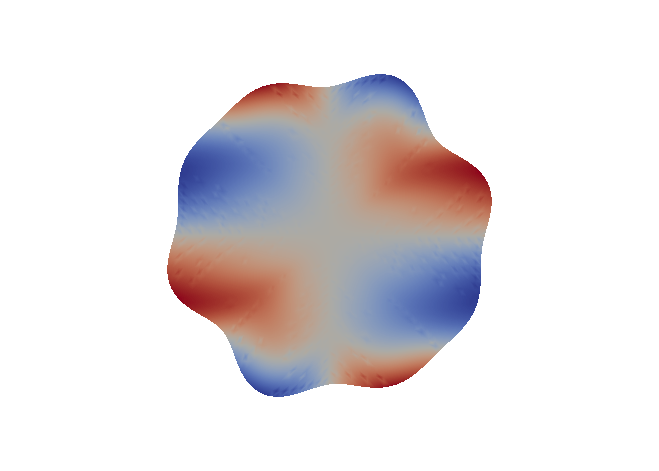} &
\includegraphics[width=1.5in]{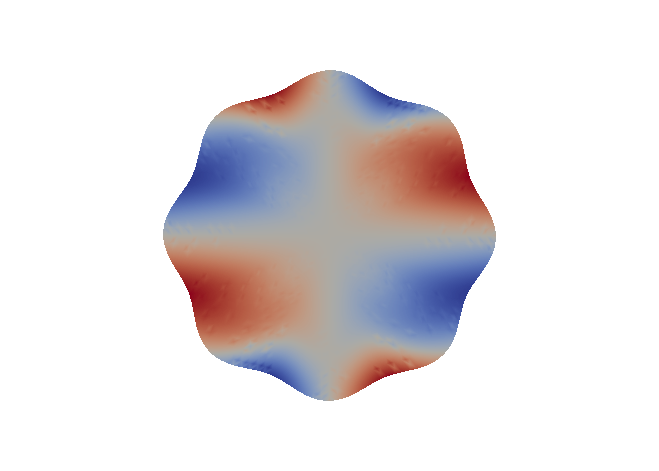} &
\includegraphics[width=1.5in]{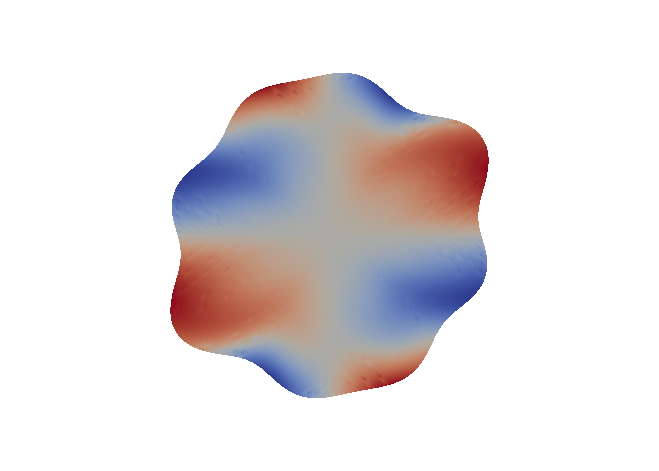} \\
\end{array}$
\end{center}
\caption{Realisations on pathwise solution on the random bulk-surface.}
\label{fig:realCoupled}
\end{figure}

We set the pull-back of the path-wise bulk solution to be given by 
\[
 \hat{u}(\omega,x) = sin(\pi x y) cos(\pi y^2) + \epsilon_{tol} \lambda(\omega) cos(\pi xy)
\]
with uniformly distributed random coefficient $\lambda\sim U(-1,1)$ and $\epsilon_{tol}=0.1$. This determines the pull-back of the path-wise surface solution $\hat{v}$ by the 
reformulated Robin boundary condition
\[
 \alpha \hat{u}(\omega) - \beta \hat{v}(\omega) + \frac{\sqrt{g(\omega)}}{\sqrt{g_{\Gamma_0}(\omega)}} G^{-1}(\omega) \nu^{\Gamma_0} \cdot \nabla \hat{u}(\omega) = 0 \quad \text{on } \Gamma_0,
\]

from which the data $f$ and $\hat{f}_{\Gamma_0}$ can then be computed.
Note that in practice, the expectation $\mathbb{E}[\hat{v}]$ and its surface derivative are approximated with Monte-Carlo sampling to sufficiently high accuracy. 
We observe the following errors and experimental order of convergence for the approximations
of the bulk $\mathbb{E}[\hat{u}] - E_M[\hat{u}_h]$ and the surface $\mathbb{E}[\hat{v}] - E_M[\hat{v}_h]$ mean solutions. 
\begin{table}[H]
\centering
\renewcommand{\arraystretch}{1.1}
\begin{tabular}{|c c| c|c c| c| c c|}
\hline
$h$	      	 & $M$	 	&    Bulk error $ E_{L^2(D_0)}$  &     $eoc(h)$ 	&    $eoc(M)$      &  Surface error $ E_{L^2(\Gamma_0)}$  &  $eoc(h)$     &    $eoc(M)$     \\ \hline 
0.27735          &       1      &       0.619144       &         -      &          -       &         5.0787          &      -        &        -	       \\ \hline	
0.156174         &       16     &       0.198298       &      1.98249   &      -0.410651   &        1.06707          &   2.71654     &      -0.562702  \\ \hline
0.0830455        &       256    &       0.0540441      &      2.05828   &      -0.468866   &        0.28356          &   2.0983      &      -0.477981  \\ \hline
0.0428353        &      4096    &       0.0152612      &      1.91003   &      -0.456067   &      0.0723061          &   2.06414     &      -0.492866  \\ \hline
\end{tabular}
\caption{Errors in $L^2(\Omega^M; L^2(D_0))$ and $L^2(\Omega^M;L^2(\Gamma_0)).$} \label{table:coupledL2}
\end{table}

\begin{table}[H]
\centering
\renewcommand{\arraystretch}{1.1}
\begin{tabular}{|c c| c|c c| c| c c|}
\hline
$h$	      	 &    $M$	&    Bulk error $E_{L^2(D_0)}$  &     $eoc(h)$ 	  &    $eoc(M)$     &  Surface error $E_{L^2(\Gamma_0)}$  &  $eoc(h)$     &    $eoc(M)$     \\ \hline 
0.27735          &     64       &      3.41133         &       -          &      -          &       15.5792           &     -         &       -	        \\ \hline
0.156174         &     256      &      2.17523         &      0.783494    &    -0.324584    &       7.85391           &     1.1926    &     -0.494068   \\ \hline
0.0830455        &    1024      &      1.08874         &      1.09584     &    -0.499252    &       4.20041           &   0.990894    &     -0.451441   \\ \hline
0.0428353        &    4096      &      0.55599         &      1.01511     &    -0.484767    &       2.12783           &    1.02727    &     -0.490574   \\ \hline
\end{tabular}
\caption{Errors in $L^2(\Omega^M; H^1(D_0))$ and $L^2(\Omega^M;H^1(\Gamma_0))$. } \label{table:coupledH1}
\end{table} 
\newpage
\bibliography{Random_Elliptic}{}
\bibliographystyle{abbrv}

\end{document}